\documentclass[12pt, reqno]{amsart}
\usepackage{verbatim}
\usepackage{graphicx, epsfig, psfrag}
\usepackage{amsfonts,amsmath,amssymb,amsbsy,amsthm}
\usepackage{bm}
\usepackage{color}
\usepackage{float}
\usepackage{hyperref}
\usepackage{mathrsfs}
\usepackage{subfigure}
\usepackage{bbm}
\usepackage[normalem]{ulem}

\newtheorem{assumption}{Assumption}

\numberwithin{equation}{section}

\usepackage{environments}

\usepackage[top=2.5cm, bottom=2.5cm, left=2.5cm, right=2.5cm]{geometry}
\renewcommand{\paragraph}[1]{\subsubsection{#1}}

\def\transpose{{\hspace{-1pt}\top}}

\def\R{\mathbb{R}}
\def\N{\mathbb{N}}

\def\<{\langle}
\def\>{\rangle}

\def\mB{{\sf B}}
\def\mF{{\sf F}}

\def\mG{{\sf G}}

\def\a{{\rm a}}
\def\c{{\rm c}}

\def\D{\nabla}

\newcommand{\triple}{(\rho \alpha\beta)}
\newcommand{\tripleTau}{(\tau\gamma\delta)}
\newcommand{\tripleSig}{(\sigma\iota\chi)}

\newcommand{\bfrho}{\mathbf{\rho}}

\definecolor{docol}{rgb}{0, 0.4, 0}
\definecolor{cocol}{rgb}{0.7, 0, 0}
\definecolor{ascol}{rgb}{0, 0, 0.7}

\def\XXint#1#2#3{{\setbox0=\hbox{$#1{#2#3}{\int}$ }
\vcenter{\hbox{$#2#3$ }}\kern-.6\wd0}}

\title{Regularity and Locality of Point Defects in Multilattices}
\author{Derek Olson and Christoph Ortner}
\thanks{DO was supported by the NSF PIRE Grant OISE-0967140. CO was supported by ERC Starting Grant 335120.}

\begin{document}

\begin{abstract}
   We formulate a model for a point defect embedded in a homogeneous
   multilattice crystal with an empirical interatomic potential interaction.
   Under a natural, phonon stability assumption we quantify
   the decay of the long-range elastic fields with increasing distance from
   the defect.

   These decay estimates are an essential ingredient in quantifying
   approximation errors in coarse-grained models and in the construction of
   optimal numerical methods for approximating crystalline defects.

\end{abstract}

\maketitle

\section{Introduction}

The mechanical and electrical properties of crystalline materials are heavily
influenced by defects in the crystalline lattice~\cite{phillips2001}.  These
range from point defects (the subject of the present work) including vacancies,
interstitials, impurities; line defects including the preeminent dislocation;
planar defects including grain boundaries; and many others including cracks and
voids.  Modeling each of these defects relies in some form on resolving the
long-range elastic fields generated by the defects.  Whether this is
accomplished via an empirical potential, continuum PDE, or multiscale method,
all of these approximations rely on \textit{decay and regularity} of the elastic
fields sufficiently far away from the defect.  For example, a key use of these decay
rates is in establishing rigorous asymptotic results for atomistic-to-continuum
methods for multilattices~\cite{olsonUnpub}.  These decay rates have long been
known in the engineering and materials community from elasticity
theory~\cite{eshelby1956,bacon1980,glebov2012} and computational
techniques~\cite{hardy1960,flinn1962,kanzaki1957,tewary1973,glebov2012}, and can
in fact be thought of as a means of classifying
defects~\cite{krivoglaz1969,gupta2002powder}.   While related mathematical results for the decay of scalar potential fields in a linearized model defined on a lattice were obtained in~\cite{rodin2002}, the first mathematical
result for proving these decay rates for an empirical atomistic model of point
defects and dislocations in Bravais lattices appeared only recently in~\cite{Ehrlacher2013} .

The present work is an extension of~\cite{Ehrlacher2013} to multilattices, which are
crystals with more than one atom per unit cell. Multilattice descriptions
allow for a much greater swath of materials to be considered including hcp
metals, diamond cubic structures, and the recently discovered two dimensional
materials, graphene and hexagonal boron-nitride, among several
others~\cite{novoselov2005}. For the sake of simplicity of presentation, we only
consider point defects in the present paper; however, there do not seem to be
major obstacles in combining the analysis for point defects presented here with
that of dislocations for Bravais lattices in~\cite{Ehrlacher2013} to also obtain
analogous results for dislocations in multilattices.

The method of obtaining these decay rates for point defects in multilattices is
similar to that of Bravais lattices; we show that the point defect solution
satisfies a linearized equation and then convert $L^1$ integrability of the
solution in Fourier space into algebraic decay in real space.  These
integrability conditions are determined from the Green's matrix of the
linearized problem. Herein lies the main difference between the Bravais lattice
and multilattice cases: the Green's matrix for a multilattice accounts for
relative shifts between atoms in each unit cell which leads to a different
structure than in the Bravais lattice case.

In Theorem~\ref{main_thm} we recover the result from the Bravais lattice case
\cite{Ehrlacher2013} that the discrete strain field decays at a rate of $r^{-d}$
where $d$ is space dimension and $r$ is the distance from the defect. The
additional new result is that the relative shifts (which are indeed also a form of
strain) also decay at a rate of $r^{-d}$.

In the process of proving this result, we also establish a convenient
connection between phonon stability and stability in a natural discrete
energy-norm, extending an analogous observation for Bravais lattices
\cite{Ehrlacher2013}. This in particular leads to a simplified proof of the fact
\cite{weinan2007cauchy} that atomistic stability (phonon stability) implies
stability of the Cauchy--Born continuum model (see also \cite{hudson2012}).

% a new proof of the fact
% that stability of the atomistic model (phonon stability) implies stability of
% the Cauchy--Born model.  This is well known for Bravais lattices, see
% e.g.~\cite{weinan2007cauchy, hudson2012}, and a proof for multilattices is
% presented in~\cite{weinan2007cauchy} for many-body potentials.  The proof
% presented here employs a different assumption of stability of the atomistic
% problem, which we show implies the assumption made in~\cite{weinan2007cauchy},
% and our proof also includes a larger class of interatomic potentials.

\subsection*{Outline}

We begin by introducing the notation for formulating the atomistic defect
problem on a multilattice and the assumptions required of the atomistic
potential in Section~\ref{model}.  Our main result, Theorem~\ref{main_thm}, is
also presented there.  We divide the proof of Theorem~\ref{main_thm} into two
sections. In Section~\ref{s:semi}, we review the required facts of the Fourier
transform and state them in the specificity and version required for the
application at hand.  Section~\ref{s:semi} also reviews the multilattice
Cauchy--Born model and proves that atomistic stability implies Cauchy--Born
stability, closely mirroring the approach of~\cite{hudson2012}.
Section~\ref{prove} subsequently provides the linearized equation that the point
defect satisfies, gives an expression for the Green's matrix associated to this
equation, and then proves our main result.

\section{Model and Main Results} \label{model}

A multilattice is a union of shifted Bravais lattices: we fix $\mF \in
\mathbb{R}^{d \times d}$ with $\det(\mF) = 1$, $d \in \{2,3\}$
and $p_0, \ldots, p_{S-1} \in \mathbb{R}^d$ with $p_0 = 0$ and
define a multilattice $\mathcal{M}$ by
\[
\mathcal{M} := \bigcup_{\alpha = 0}^{S-1} \left(\mF \mathbb{Z}^d + p_\alpha \right).
\]
The set $\mF \mathbb{Z}^d$ is a Bravais lattice and comprises the set of \textit{sites} in the lattice; we denote it by $\mathcal{L} := \mF \mathbb{Z}^d$.  (The conditions $\det(\mF) = 1$ and $p_0 = 0$ are merely for convenience of notation and do not restrict the generality of the analysis.) Deformations and displacements of atoms of species $\alpha$ at site $\xi \in \mathcal{L}$ are, respectively, denoted by $y_\alpha(\xi):\mathbb{R}^d \to \mathbb{R}^n$ and $u_\alpha(\xi):\mathbb{R}^d \to \mathbb{R}^n$, where we permit $n = d$ or $n = d + 1$ when $d = 2$.    The set of all $S$ deformations and displacements are denoted by $\bm{y}(\xi):\mathcal{L}^S \to \mathbb{R}^n$ and $\bm{u}(\xi):\mathcal{L}^S \to \mathbb{R}^n$ where $\mathcal{L}^S = \mathcal{L} \times \cdots \times \mathcal{L}$.

To describe interactions between atoms, we define a finite difference notation (on either deformations or displacements) indexed by
\begin{align*}
   D_{\triple}\bm{u}(\xi) &:= u_\beta(\xi + \rho) - u_\alpha(\xi),
   \qquad \text{where } \\
   \triple &\in \mathcal{L} \times \{0, \ldots, S-1\} \times \{0, \ldots, S-1\}.
\end{align*}
The collection of finite differences describing the interaction of a
site $\xi$ is denoted by
\[
D\bm{u}(\xi) := \left(D_{\triple}\bm{u}(\xi)\right)_{\triple \in \mathcal{R}},
\]
where $\mathcal{R} \subset \mathcal{L} \times \{0, \ldots, S-1\} \times \{0,
\ldots, S-1\} \setminus \bigcup_{\alpha = 0}^{S-1}\{ (0\alpha\alpha)\}$ is a finite interaction range satisfying the conditions
\begin{align}
   & {\rm span}\{ \rho \,|\, (\rho\alpha\alpha) \in \mathcal{R} \} = \R^d \text{ for all $\alpha \in \mathcal{S}$}, \label{cond1} \\
   & (0\alpha\beta) \in \mathcal{R} \quad \text{for all $\alpha \neq \beta \in \mathcal{S}$ } \label{cond2}.
\end{align}
These two conditions, as well as a further condition \eqref{assumption:mesh}
are made for convenience of notation but do not restrict generality since
we can always enlarge the interaction range $\mathcal{R}$ to satisfy them.
For future reference, we denote the projection of $\mathcal{R}$ onto the
lattice component by
\begin{equation*}
   \mathcal{R}_1 := \big\{ \rho \in \mathcal{L} \,|\, \exists \triple \in \mathcal{R} \big\}
\end{equation*}
and finite differences on individual displacements, $u_\alpha$, by
\begin{equation*}
D_\rho u_\alpha (\xi) := u_\alpha(\xi + \rho) - u_\alpha(\xi), \quad Du_\alpha(\xi) := \big(D_\rho u_\alpha(\xi)\big)_{\rho \in \mathcal{R}_1}.
\end{equation*}

We assume that the atomistic energy may be written (formally) as a sum of site potentials,
\[
\hat{\mathcal{E}}^\a(\bm{y}) := \sum_{\xi \in \mathcal{L}} \hat{V}_{\xi}(D\bm{y}(\xi)),
\]
where the site potential, $\hat{V}_\xi$, is assumed to satisfy:
\def\Rdef{R_{\rm def}}
\begin{enumerate}
\item[V.1] There exists $\Rdef > 0$ such that for all $|\xi| \geq \Rdef$, $\hat{V}_\xi \equiv \hat{V}$ does not depend on $\xi$.  This assumption is valid for point defects located near the origin.
\end{enumerate}

For the atomistic energy functional to be well-defined (i.e. finite), we will
consider an energy difference functional defined on displacements, $\bm{u}$,
from a reference state, $\bm{y}(\xi)$, which is defined differently depending on
whether $d = n$ or not. When $d = n$, which models bulk crystals, we set
\[
y_\alpha(\xi) = \xi +p_\alpha,
\]
where each $p_\alpha \in \mathbb{R}^d$. If $d = 2$ and $n = 3$, which is the case when modeling monolayer materials such as graphene, then
we set
\[
y_\alpha(\xi) = \begin{pmatrix}\xi \\ 0\end{pmatrix} +
\begin{pmatrix}p_\alpha \\ 0 \end{pmatrix}.
\]
In the latter case, we will
drop the third component being equal to zero under the understanding that $\xi,
p_\alpha \in \mathbb{R}^d$ are considered as elements in $\mathbb{R}^n$ in this
fashion. Thus, $\xi, p_\alpha$ may either denote vectors in $\R^d$ or $\R^n$,
but it will always be clear from the context what we mean.

This energy difference functional is defined by
\begin{equation}\label{at_energy}
   \mathcal{E}^\a(\bm{u}) :=
   \sum_{\xi \in \mathcal{L}} V_\xi(D\bm{u}(\xi)), \quad V_\xi(D\bm{u}) := \hat{V}_{\xi}(D\bm{y} + D\bm{u}) - \hat{V}(D\bm{y}).
\end{equation}

An auxiliary energy functional needed in the subsequent analysis is the energy of the homogeneous (defect-free) lattice
\[
   \mathcal{E}_{\rm hom}^\a(\bm{u}) :=  \sum_{\xi \in \mathcal{L}} V(D\bm{u}(\xi)),
   \qquad V(D\bm{u}) := \hat{V}(D\bm{y}+D\bm{u}) - \hat{V}(D\bm{y}).
\]

% We will likewise often be working with derivatives of these energy functionals for which we need a convenient notation for derivatives of the site potential with respect to an element $\triple \in \mathcal{R}$.
Arguments of the site potentials are indexed by $\triple \in \mathcal{R}$.
Given $\triple, \tripleTau \in \mathcal{R}$ and $\bm{g} = (\bm{g}_{\triple})_{\triple \in \mathcal{R}} \in (\R^n)^{\mathcal{R}}$, we will denote derivatives of $V_{\xi}$ (or $\hat{V}_\xi$) by
\begin{align*}
[V_{\xi,\triple}(\bm{g})]_{i} :=~& \frac{\partial V_\xi(\bm{g})}{\partial \bm{g}_{\triple}^i }, \quad i = 1,\ldots, n, \\
V_{\xi,\triple}(\bm{g}) :=~& \frac{\partial V_\xi(\bm{g})}{\partial \bm{g}_{\triple}}, \\
[V_{\xi,\triple\tripleTau}(\bm{g})]_{ij} :=~& \frac{\partial^2 V_\xi(\bm{g})}{\partial \bm{g}_{\tripleTau}^j \partial \bm{g}_{\triple}^i},  \quad i,j = 1,\ldots, n, \\
V_{\xi,\triple\tripleTau}(\bm{g}) :=~& \frac{\partial^2 V_\xi(\bm{g})}{\partial \bm{g}_{\tripleTau} \partial \bm{g}_{\triple}},
\end{align*}
with higher order derivatives defined analogously. Moreover, it will later be notationally convenient to consider derivatives with $\triple \notin \mathcal{R}$, in which case
\[
V_{\xi,\triple}(\bm{g}) = 0,
\]
and so on for higher order derivatives.  With this notation, the site potential is additionally assumed to satisfy the following differentiability assumption:
\begin{enumerate}
\item[V.2] Each $\hat{V}_\xi:(\mathbb{R}^n)^\mathcal{R} \to \mathbb{R}$ is four times
         continuously differentiable with uniformly bounded derivatives.
\end{enumerate}

The function space on which $\mathcal{E}^\a$ will be defined is a quotient space of a set of discrete displacements having a finite ``energy'' norm,
\[
\| \bm{u}\|_{\a_1}^2 := \sum_{\xi \in \mathcal{L}} |D\bm{u}(\xi)|_\mathcal{R}^2,
 \qquad \text{where }  |D\bm{u}|_{\mathcal{R}}^2 := \sum_{\triple \in \mathcal{R}} |D_{\triple} \bm{u}(\xi)|^2.
\]
In view of \eqref{cond1} and \eqref{cond2}, $\|\bm{u}  \|_{\a_1} = 0$
if and only if there exists $v \in \R^n$ such that $u_\alpha = v$ for all
$\alpha = 0, \dots, S-1$.

Because of the translation invariance of $\mathcal{E}^\a(\bm{u})$ we will define it
on the quotient space
\begin{align*}
   \bm{\mathcal{U}} := \mathcal{U}/ \mathbb{R}^n, \qquad \text{where}
   \quad \mathcal{U} := \left\{\bm{u}:\mathcal{L}^S  \to \mathbb{R}^n, \|\bm{u}\|_{\a_1} < \infty   \right\}.
\end{align*}
Proving that $\mathcal{E}^\a$ is well defined on this space will rely on density of the space of compactly supported test functions, $\bm{\mathcal{U}}_0$, defined by
\begin{align*}
\mathcal{U}_0 :=~& \left\{\bm{u} \in \mathcal{U} : Du_0, u_\alpha - u_0 \,
               \mbox{have compact support for each $\alpha$} \right\}, \\
\bm{\mathcal{U}}_0 :=~& \mathcal{U}_0 / \mathbb{R}^n.
\end{align*}
It is straightforward to establish that $\bm{\mathcal{U}}_0$ is dense in
$\bm{\mathcal{U}}$; see Lemma~\ref{lem:dense} for a proof.

It is clear that $\mathcal{E}^\a$ and $\mathcal{E}^\a_{\rm hom}$ are
well-defined on $\bm{\mathcal{U}}_0$ since only finitely many summands will be
nonzero in this case.  Our choice of function space, $\bm{\mathcal{U}}$, is
justified in the following theorem, and we will prove below in Lemma
\ref{th:equil_shift_cor} that the hypothesis of the theorem is in fact
equivalent to the lattice energy per unit volume being minimized over the
internal shifts. This implies, in particular that \eqref{ostrich1} is
straightforward to enforce in practical computations.

\begin{theorem}\label{well_defined}
If the reference configuration $\bm{y}$ with $y_\alpha(\xi) = \xi + p_\alpha$ is an equilibrium of the defect free energy, that is,
\begin{equation}\label{ostrich1}
\sum_{\xi \in \mathcal{L}} \sum_{\triple \in \mathcal{R}} \hat{V}_{,\triple}(D\bm{y}(\xi)) \cdot D\bm{v}(\xi) = 0, \quad \forall \, \bm{v} \in \bm{\mathcal{U}}_0,
\end{equation}
then the energy functionals, $\mathcal{E}^\a_{\rm hom}(\bm{u})$ and $\mathcal{E}^\a(\bm{u})$, can be uniquely extended to continuous functions on $\bm{\mathcal{U}}$ which are well-defined and ${\rm C}^3$ on $\bm{\mathcal{U}}$.
\end{theorem}

\begin{remark}\label{rem:renormalisation}
   The proof of Theorem~\ref{well_defined} is based on the idea that,
   for $\bm{u} \in \bm{\mathcal{U}}_0$,
   $\mathcal{E}^\a_{\rm hom}(\bm{u}) = \bar{\mathcal{E}}^\a_{\rm hom}(\bm{u}))$, where
   \[
      \bar{\mathcal{E}}^\a_{\rm hom}(\bm{u}) := \sum_{\xi \in \mathcal{L}} \big[V(D\bm{u}(\xi)) - \sum_{\triple \in \mathcal{R}}V_{,\triple}(D\bm{y}(\xi))\cdot D_{\triple}\bm{u}(\xi)\big].
   \]
   While $\mathcal{E}^\a_{\rm hom}$ is well-defined only if $D\bm{u} \in \ell^1$,
   $\bar{\mathcal{E}}^a_{\rm hom}$ is also well-defined for $D\bm{u} \in \ell^2$.
   However, since $\bar{\mathcal{E}}^\a_{\rm hom}$ is the unique continuous extension
   of $\mathcal{E}^\a_{\rm hom}$ from $\bm{\mathcal{U}}_0$ to $\bm{\mathcal{U}}$ we will
   continually use $\mathcal{E}^\a(\bm{u})_{\rm hom}$ (and $\mathcal{E}^\a(\bm{u})$) in lieu of $\bar{\mathcal{E}}^\a(\bm{u})$ (and an analogously defined $\bar{\mathcal{E}}^\a$).

\end{remark}

\medskip

Having established that $\mathcal{E}^\a(\bm{u})$ is well-defined on
the natural energy space $\bm{\mathcal{U}}$, we are interested in
the force equilibrium problem
\begin{equation} \label{eq:equlibrium}
   \<\delta \mathcal{E}^\a(\bm{u}^\infty), \bm{v}\> = 0,
   \qquad \forall \, \bm{v} \in \bm{\mathcal{U}}_0.
\end{equation}
Two important special cases are local minima (stable equilibria) and
index-1 saddles (transition states between stable equilibria). In the present
work we will not go into details about these specific problems but focus on
the regularity of equilibria, i.e., solutions to \eqref{eq:equlibrium}.

Our analysis requires only the following standing assumption:

\begin{assumption}\label{coercive}
   (1) The reference configuration, $\bm{y}$, with $y_\alpha(\xi) = \xi + p_\alpha$
   is a stable equilibrium of $\mathcal{E}^\a_{\rm hom}$, that is, in
   addition to \eqref{ostrich1} we require that there exists $\gamma_\a > 0$
   such that
   \begin{equation} \label{eq:stab-hom}
      \<\delta^2\mathcal{E}^\a_{\rm hom}(0)\bm{v},\bm{v}\> \geq~ \gamma_\a\|\bm{v}\|_{\a_1}^2 , \quad \forall \, \bm{v} \in \bm{\mathcal{U}}_0.
   \end{equation}

   (2) There exists a solution $\bm{u}^\infty \in \bm{\mathcal{U}}$
   to \eqref{eq:equlibrium}.
\end{assumption}

\begin{remark}
   Note that Assumption~\ref{coercive} imposes no additional structure on
   solutions $\bm{u}^\infty$ but only on the reference state.
   Physically, the requirement \eqref{eq:stab-hom} is a minimal assumption on
   the stability of lattice waves, called {\em phonon stability}, made throughout the
   solid state physics literature~\cite{born1954}, and is almost universally
   reasonable.
   % It is related, as we will show later, to the notion of  as it ensures the phonon frequencies are real.

	Moreover, one can readily show (see Lemma~\ref{freeStable} in the appendix
	or~\cite[Section 2.2]{Ehrlacher2013} for a related result for Bravais
	lattices) that, if there exists {\em any} stable equilibrium of
	$\mathcal{E}^\a$, then \eqref{eq:stab-hom} holds as well.
\end{remark}

\medskip

The decay rates we prove in Theorem~\ref{main_thm} below are formulated
in terms of the finite difference notation
\begin{align*}
   D_{\rho}u_\alpha(\xi) &:= u_\alpha(\xi + \rho) - u_\alpha(\xi)
   \qquad \text{for }
   \rho \in \mathcal{L}, \quad \alpha \in \mathcal{S}, \qquad \text{and} \\
   D_{\bm{\rho}} u_\alpha(\xi) &:= D_{\rho_1}D_{\rho_2}\cdots D_{\rho_k} u_\alpha(\xi)
   \qquad \text{for } \bm{\rho} = (\rho_1, \dots, \rho_k) \in \mathcal{L}^k.
\end{align*}
We interpret the finite differences $D_\rho u$ as an ``atomistic strain'' and
the higher order differences as discrete strain gradients.
%  We define the length of
% $\bm{\rho} = \rho_1\rho_2 \cdots \rho_k$ by $|\bm{\rho}| = k$.  With this
% notation in hand, our main result is as follows.

\begin{theorem}[Decay of Displacements and Shifts]\label{main_thm}
   Suppose that Assumption~\ref{coercive} holds and set
   $U^\infty = u_0^\infty, p_\alpha^\infty = u_\alpha^\infty - u_0^\infty$. Then
   \begin{equation}\label{decay1}
      \begin{split}
      \big|D_{\bm{\rho}} U^\infty(\xi)\big| \lesssim~& (1 + |\xi|)^{1-d-j},
            \quad \forall \bm{\rho} \in (\mathcal{R}_1)^j, 1 \leq j \leq 3, \quad \text{and} \\
      \big|D_{\bm{\rho}} p_\alpha^\infty(\xi)\big| \lesssim~& (1 + |\xi|)^{-d-j},
          \quad \forall \bm{\rho} \in (\mathcal{R}_1)^j, 0 \leq j \leq 2.
      \end{split}
   \end{equation}
\end{theorem}

In the statement of the theorem, we have used the modified Vinogradov notation $A \lesssim B$ to mean there exists a constant $c > 0$ such that $A \leq cB$.  The implied constant here (and throughout the remainder of the paper) is allowed to depend upon the interatomic potential, interaction range, and stability constant $\gamma_\a$.

The rest of the paper is devoted to proving Theorem~\ref{main_thm}. We will
first exhibit a linearized equation which $\bm{u}^\infty$ satisfies and prove
decay rates for the Green's function associated with this linearized problem.  The
key point in proving the decay rates for the Green's function will be connecting
$L^1$ integrability of a function's Fourier transform with $L^\infty$ decay of
the original function.  Meanwhile, the $L^1$ estimates in Fourier space are
obtained by comparing the atomistic Green's function with the
\textit{Cauchy--Born} continuum Green's function.

\begin{remark}[Other point defects]
   Although superficially we have only included an impurity defect in defining our model energy,
   Theorem~\ref{main_thm} actually applies to arbitrary point defects, including
   for example vacancies and interstitials.

   To see this, consider a defective lattice, $\mathcal{L}^{\rm def}$, with a ``defect core radius,'' $R_{\rm def}$,
   such that $\mathcal{L} \setminus B_{R_{\rm def}} = \mathcal{L}^{\rm def} \setminus B_{R_{\rm def}}$,
   and let $u^{\rm def} : \mathcal{L}^{\rm def} \to \R^n$ be an equilibrium
   of an energy functional analogous to $\mathcal{E}^\a$, in particular employing
   the same homogeneous potential $V$ in $\mathcal{L} \setminus B_{R_{\rm def}}$.
   Then, projecting $u^{\rm def}$ to a displacement $u : \mathcal{L} \to \R^n$
   with $u(\xi) = u^{\rm def}(\xi)$ in $\mathcal{L} \setminus B_{R_{\rm def}}$, we obtain a new
   displacement satisfying
   \[
   \frac{\partial \mathcal{E}^\a_{\rm hom}(u)}{\partial u_\alpha(\xi)} = 0,
      \qquad \forall \, |\xi| \geq R_{\rm def}',
   \]
   for some $R_{\rm def}' \geq 0$
   but potentially non-zero forces in $B_{R_{\rm def}'}$. By
   defining $V_\xi(Du) = V(Du) + \bm{g}_\xi \cdot Du$
   with suitable $\bm{g} \in (\R^n)^{\mathcal{R}}$ for $\xi \in B_{R_{\rm def}}$,
   we are put precisely in the context of Theorem~\ref{main_thm}, and thus
   the decay estimates again apply.
\end{remark}

\section{Preliminaries}\label{s:semi}

In this section we collect a range of auxiliary results that are required in
the proof of Theorem~\ref{main_thm}.

\subsection{Continuous interpolants of lattice functions}
It is often useful to identify lattice functions with continuous
interpolants. To define these, we divide the unit cell $\mF[0,1]^d$ into
simplices (triangles in $2D$ and tetrahedra in $3D$) so that each vertex of a
simplex is one of the vertices of $\mF[0,1]^d$. A simplicial decomposition,
$\mathcal{T}_\a$, of $\mathcal{L}$ is completed by performing the same
decomposition on the translated cells $\xi + \mF[0,1]^d$ for $\xi \in
\mathcal{L}$. Note that this can be done in such a way that $\mathcal{T}_\a$
is {\em regular}.

For $u: \mathcal{L} \to \mathbb{R}^n$, we then denote the continuous
interpolant of $u$ with respect to $\mathcal{T}_\a$ by $Iu$.  We will also write
 $I\bm{u} = (Iu_\alpha)_{\alpha = 0}^{S-1}$. By possibly enlarging $\mathcal{R}$
 we may assume without loss of generality that
 \begin{equation} \label{assumption:mesh}
     \text{if } {\rm conv}\{ \xi, \xi+\rho \} \text{ is an edge of $\mathcal{T}_\a$,
     then } \rho \in \mathcal{R}_1.
 \end{equation}

This construction gives rise to a natural alternative norm
for multilattice displacements,
\[
   \|\bm{u}\|_{\a_2} := \|\nabla I u_0\|_{L^2(\mathbb{R}^d)} + \sum_{\alpha = 0}^{S-1}\|Iu_\alpha - Iu_0\|_{L^2(\mathbb{R}^d)},
\]
which turns out to be equivalent to $\|\cdot\|_{\a_1}$.

\begin{lemma}
   The norms, $\|\cdot \|_{\a_1}$ and $\| \cdot \|_{\a_2}$, are equivalent on
   the set of multilattice displacements $\bm{u}:\mathcal{L}^S \to
   \mathbb{R}^n$.
\end{lemma}
\begin{proof}
   From \eqref{assumption:mesh} it is clear that $\|\cdot\|_{\a_2} \lesssim \|
   \cdot \|_{\a_1}$. To prove the opposite, let $\omega := \bigcup \{ T \in
   \mathcal{T}_\a | T \cap \mathcal{R}_1 \neq \emptyset \}$
   (the minimal patch of elements $T$ covering the interaction neighbourhood),
   then
    \begin{equation} \label{eq:local norm equiv}
       |D\bm{u}(\xi)|^2 \leq C\Big( \| \D I u_0 \|_{L^2(\xi+\omega)}^2
               + {\textstyle \sum_\alpha} \| Iu_\alpha - Iu_0 \|_{L^2(\xi + \omega)}^2 \Big).
    \end{equation}
   This follows from the fact that both sides of the inequality involve
   only finitely many degrees of freedom and, if the right-hand side
   vanishes, then so does the left-hand side.

   The stated result now follows by summing \eqref{eq:local norm equiv}
   over $\mathcal{L}$.
   % let $\mathcal{N}(T)$ be the set of vertices of
   % $T$.  Then both $\|\cdot \|_{\a_1}$ and $\| \cdot \|_{\a_2}$ are semi-norms
   % on the space of displacements, $\bm{u}$, restricted to $\mathcal{N}(T)$ owing
   % to \eqref{assumption:mesh}.  Furthermore, the kernel of each seminorm is the
   % set of constant functions, $u_\alpha(\xi) = c$ for all $\alpha$ and all $\xi
   % \in \mathcal{N}(T)$.  This implies $\|\cdot \|_{\a_1}$ and $\| \cdot
   % \|_{\a_2}$ are locally equivalent and summing gives the global statement.
\end{proof}

\subsection{Semi-discrete Fourier transform for multilattices}

The first Brillouin zone, $\mathcal{B}$, is defined as the Voronoi cell
associated with the origin in the dual lattice, $\mB\mathbb{Z}^d$, with $\mB =
\mF^{-\transpose}$. For a lattice function $u:\mathcal{L} \to \mathbb{R}^n$, the
semidiscrete Fourier transform, and its inverse are, respectively, defined by
\begin{align*}
\hat{u}(k) &= \sum_{\xi \in \mathcal{L}} e^{-2\pi i \xi \cdot k} u(\xi), \quad \mbox{for $k \in \mathcal{B}$,} \qquad
\check{v}(\xi) = \int_{\mathcal{B}} e^{2\pi i \xi \cdot k} v(k)\, dk, \quad \mbox{for $\xi \in \mathcal{L}$.}
\end{align*}
As usual, the discrete Fourier transform is well-defined for
$\ell^1(\mathcal{L})$ functions and otherwise defined through continuity.

The semidiscrete Fourier transform (and its inverse) possesses the usual
transform properties; for the task at hand, the most important of these is the
connection between $L^1$ integrability of a function's (semidiscrete) Fourier
transform and its derivatives and the $L^\infty$ decay of the original function
and its derivatives.

As the first Brillouin zone is a finite domain, and many of the fields involved will be either smooth or only singular at the origin, we will be most concerned with the behavior of the Fourier transform near the origin.  For this reason, we introduce a ``big O notation''
\[
   f(k) = \mathcal{O}(g(k))
   \quad \text{if and only if} \quad
   \exists\, C > 0 \text{ s.t. } |f(k)| \leq C |g(k)| \text{ for all } k \in \mathcal{B},
\]
which is modified from the standard notation in that we require the upper
bound in the entire domain of definition $\mathcal{B}$.

\begin{theorem}\label{decay_thm}
Suppose that $f:\mathcal{L} \to \mathbb{R}^n$ is a function such that $\hat{f}, \nabla \hat{f}, \cdots, \nabla^m \hat{f} \in L^1(\mathcal{B})$, then
\[
   |f(\xi)| \lesssim~ (1 + |\xi|)^{-m}, \quad \forall \, \xi \in \mathcal{L}.
\]
\end{theorem}
\begin{proof}
   The proof uses standard techniques and while related results exist throughout
   the literature~\cite{trefethen2000,rudin1987}, we were unable to find a
   statement of the specificity that we require here, hence we include a
   proof for convenience and completeness.

Let $\gamma$ be any multiindex with $|\gamma| \leq m$.  Then using the fact that
\[
(\partial_\gamma \hat{f})^{\vee}(\xi) = \int_{\mathcal{B}} e^{2\pi i \xi \cdot k} \partial_\gamma \hat{f}(k)\, dk = (2\pi i)^{|\gamma|} \xi^\gamma \int_{\mathcal{B}} e^{2 \pi i \xi \cdot k} \hat{f}(k) = (2\pi i)^{|\gamma|} \xi^\gamma f(\xi)
\]
and
\[
\|f\|_{\ell^\infty} \leq \|\hat{f}\|_{L^1},
\]
we see that
\[
(2\pi)^{|\gamma|} \|\xi^\gamma f(\xi)\|_{\ell^\infty} = \|(2\pi i)^{|\gamma|} \xi^\gamma f(\xi)\|_{\ell^\infty} \leq~ \|\partial_\gamma \hat{f}\|_{L^1}.
\]
This in turn implies $|\xi|^m f(\xi)$ is bounded.
\end{proof}

Since we will later employ Taylor expansions in Fourier space along with operating with
finite differences, a useful (and almost immediate) corollary of this result is
the following.

\begin{corollary}\label{lem:space1A}
   Let $f:\mathcal{L} \to \mathbb{R}^n$, and assume there is an integer $s \geq -1$ such that $\nabla^j\hat{f}(k) = \mathcal{O}(k^{s-j})$ for all nonnegative integers $j$. Then
\begin{align*}
   % |f(\xi)| \lesssim~& (1 + |\xi|)^{-s-d+1}, \quad \xi \in \mathcal{L},
   % \quad \text{and} \\
   |D_{\bm{\rho}} f(\xi)| \lesssim~& (1 + |\xi|)^{-s-d+1-t}
   \qquad \text{for} \quad  \xi \in \mathcal{L}, \quad \bm{\rho} \in (\mathcal{R}_1)^t, \quad t \geq 0.
\end{align*}
\end{corollary}

\begin{proof}
   % Let $R > 0$ be such that $B_R(0)$ is the smallest ball of radius $R$ centered
   % about the origin containing $\mathcal{B}$, then we have
   % \begin{align*}
   %    \int_{\mathcal{B}} |\nabla^j \hat{f}(k)|\, dk
   %    \leq~& \int_{B_R(0)} |k|^{s-j} \, dk
   %    \lesssim \int_{r = 0}^R r^{s-j+d-1}\,dr.
   % \end{align*}
   % Therefore, $\nabla^j \hat{f}(k) \in L^1(\mathcal{B})$, provided that $s - j +
   % d - 1 > -1$, or $s - j + d > 0$.  Thus, the largest integer $j$ such that
   % $\nabla^j\hat{f}(k)$ remains in $L^1(\mathcal{B})$ is $j = s + d -1$. From
   % Theorem~\ref{decay_thm}, it immediately follows that
   % \[
   %    |f(x)| \lesssim~ (1 + |x|)^{-s-d+1}.
   % \]
   % This proves the stated decay estimate for $t = 1$.
   %
   Let $\bm{\rho} = \rho_1\cdots\rho_t \in (\mathcal{R}_1)^t$.
   By Theorem~\ref{decay_thm}, to prove the stated decay, it is sufficient to
   show that
   \[
      \nabla^{j} \widehat{D_{\bm{\rho}} f}(k) \in L^1(\mathcal{B})
      \qquad \text{for} \quad  j = 0, \ldots, t+s+d-1.
   \]

   To that end we first note that
   \begin{align*}
      \widehat{D_{\bm{\rho}} f}(k) &= (e^{2\pi ik\cdot \rho_1} - 1)(e^{2\pi ik\cdot \rho_2} - 1)\cdots(e^{2\pi ik\cdot \rho_t} - 1)\hat{f}.
   \end{align*}
   Next, we observe that
   \[
      \nabla^j \big((e^{2\pi ik\cdot \rho_1} - 1)(e^{2\pi ik\cdot \rho_2} - 1)\cdots(e^{2\pi ik\cdot \rho_t} - 1)\big) = \mathcal{O}(k^{t-j}) \quad \mbox{for $j \geq 0$,}
   \]
   while $\nabla^j\hat{f}(k) = \mathcal{O}(k^{s-j})$ for $j \geq 0$ by
   assumption. Therefore,
   \begin{align*}
      \int_{\mathcal{B}} \left|\nabla^j \big( (e^{2\pi ik\cdot \rho_1} - 1)(e^{2\pi ik\cdot \rho_2} - 1)\cdots(e^{2\pi ik\cdot \rho_t} - 1)\big)\hat{f}\right|\, dk \lesssim~ \int_{B_R(0)} |k|^{t+s-j}\, dk.
   \end{align*}
   Hence, $\widehat{D_{\bm{\rho}} f}(k) \in L^1(\mathcal{B})$ provided $t+s-j + d-1 > -1$.  This last statement is true for $0 \leq j \leq t+s + d -1$, and
   we obtain the desired result.
\end{proof}

\subsection{The multilattice Cauchy--Born model}

The next ingredient for our analysis is the Cauchy--Born energy functional.  We
will later compare the Hessian of a linearized atomistic model with that of the
Cauchy--Born Hessian in order to glean information about the atomistic Green's
matrix from the Cauchy--Born Green's matrix.   The Cauchy--Born energy functional
was originally proposed by Cauchy for Bravais lattices~\cite{cauchy} and was
later extended to multilattices~\cite{born1954}.  The fundamental idea behind
the original Cauchy rule for Bravais lattices was that the atomistic and
continuum kinematics could be related by assuming that a continuum strain
affected the atomistic model by straining the lattice basis vectors as if they
were part of the continuous medium~\cite{cauchy}.  The adaptation of this to
multilattices proceeded by further assuming that the relative shifts between
atoms inside each unit cell were equilibrated~\cite{born1954}.

For our purposes, we will introduce both the classical Cauchy--Born energy for
multilattices, and a variant used in~\cite{koten2013,olsonUnpub}, which maintains
the relative shifts in each unit cell as degrees of freedom in the energy
functional.  Throughout this section, we will employ the displacement-shift
kinematic description of the multilattice.  That is, we define a base
displacement at each Bravais lattice site by $U(\xi) = u_0(\xi)$ and then
define the relative shifts within each unit cell by $p_\alpha(\xi) =
u_\alpha(\xi) - u_0(\xi)$ and $\bm{p} = (p_0, \ldots, p_{S-1})$.  In this
notation, the ``non-classical'' variant of the Cauchy--Born strain energy
density functional is defined for $\mG \in \mathbb{R}^{n \times d}$ and $\bm{p} \in \mathbb{R}^n$ by
\begin{align*}
\hat{W}(\mG, \bm{p}) :=~& \hat{V}\big((\mG \rho + p_\beta - p_\alpha)_{\triple \in \mathcal{R}}\big),
\end{align*}
and for $U \in {\rm C}^1(\mathbb{R}^d, \mathbb{R}^n)$ and $p_\alpha \in {\rm C}^0(\mathbb{R}^d, \mathbb{R}^n)$ by
\begin{align*}
W( (U,\bm{p})) :=~& V\big((\nabla_\rho U + p_\beta - p\alpha )_{\triple \in \mathcal{R}}\big).
\end{align*}
The Cauchy--Born continuum energy is then, formally, defined by
\[
\mathcal{E}^\c(U,\bm{p}) =~ \int_{\mathbb{R}^d} W( (U,\bm{p}))\, dx.
\]

The classical variant of the Cauchy--Born rule~\cite{born1954} additionally
enforces that the shifts in each unit cell are equilibrated in the sense that
the energy in each unit cell is minimized.  Thus, it defines a strain energy
density functional on $\mathbb{R}^{n \times d}$ by
\begin{equation}  \label{eq:classical-cb}
   \bar{W}(\mG) := \min_{\bm{p} \in (\R^d)^S} \hat{V}\big((\mG \rho + p_\beta - p\alpha )_{\triple \in \mathcal{R}}\big).
\end{equation}

A useful relation between the classical Cauchy--Born rule and the atomistic
model is that minimizing $\hat{V}$ with respect to the shifts in each
unit cell is equivalent to the equilibrium condition that we used in
Theorem~\ref{well_defined} to show that $\mathcal{E}^\a$ is well-defined.

\begin{lemma} \label{th:equil_shift_cor}
   Recall the multilattice is defined by $\mathcal{M} := \bigcup_{\alpha = 0}^{S-1} \left(\mF \mathbb{Z}^d + p_\alpha \right)$, and let $\bm{y}$ be the reference deformation defined by
    $y_\alpha(\xi) = \xi + p_\alpha$. If $d = n$, then set $\mG = I_{d\times d} \in \mathbb{R}^{d \times d}$, and if $d \neq n$, set $\mG = \begin{pmatrix} I_{d\times d} \\ \bm{0} \end{pmatrix}$ and consider each $p_\alpha \in \mathbb{R}^d$ to be in $\mathbb{R}^n$ via $p_\alpha = \begin{pmatrix} p_\alpha \\ 0\end{pmatrix}$.  Then the following two
    conditions are equivalent:
   \begin{align*}
      \partial_{\bm{p}} \hat{W}(\mG, \bm{p})
      =
      \partial_{\bm{p}} \hat{V}\big((\mG \rho + p_\beta - p_\alpha )_{\triple \in \mathcal{R}}\big) &= 0, \qquad \text{and} \\
   \sum_{\xi \in \mathcal{L}} \sum_{\triple \in \mathcal{R}} \hat{V}_{,\triple}(D\bm{y}(\xi)) \cdot D_{\triple} \bm{v}(\xi) &= 0,
   \qquad \forall \, \bm{v} \in \bm{\mathcal{U}}_0.
   \end{align*}
\end{lemma}
\begin{proof}
   We define the test function $\bm{v}$ by $v_\gamma(\zeta) = 1$,
   $v_\gamma(\xi) = 0$ for $\xi \neq \zeta$, and $v_\beta(\xi) = 0$
   for all $\beta \neq \gamma$. Then a straightforward computation (see Appendix~\ref{sec:cb_at}) yields
   \begin{equation} \label{eq:prf-equil_shift_cor}
      \< \mathcal{E}^\a_{\rm hom}(\bm{0}), \bm{v} \>
         = \partial_{p_\gamma} \hat{W}(\mG, \bm{p}),
   \end{equation}
   which implies that the result.
\end{proof}

Another relation between the Cauchy--Born rule and the atomistic model is the
fact that the stability assumption, Assumption~\ref{coercive}, implies an
analogous stability condition for the Cauchy--Born energy functional.

For the purpose of proving this auxiliary result, we temporarily consider a
finite continuum domain $\Omega = (-1/2, 1/2]^d$, a corresponding finite
atomistic domain
\[
   \Omega_\epsilon :=
   \left\{-1/2 + \epsilon, -1/2+2\epsilon, \cdots, 1/2-\epsilon, 1/2\right\}^d,
\]
associated atomistic and continuum energies, and appropriate norms defined by
\begin{align*}
\mathcal{E}^\a_\epsilon(\bm{u}^\epsilon) :=~& \epsilon^{d} \sum_{\xi \in \Omega_\epsilon} V_\xi(D^\epsilon \bm{u}^\epsilon) \quad \mbox{where} \quad D^\epsilon_{\triple}\bm{u}^\epsilon(\xi) := \frac{u_\beta^\epsilon(\xi + \epsilon\rho) - u_\alpha^\epsilon(\xi)}{\epsilon} \\
\|\bm{u}^\epsilon\|_{\a,\epsilon}^2 =~& \|\nabla I_\epsilon u^\epsilon_0\|^2_{L^2(\Omega)} + \sum_{\alpha=0}^{S-1} \epsilon^{-2}\| I_\epsilon u_\alpha^\epsilon - I_\epsilon u^\epsilon_0\|^2_{L^2(\Omega)}, \qquad \text{and} \\
\mathcal{E}^\c_{\Omega}(U, \bm{p}) =~& \int_{\Omega} V\Big((\nabla_\rho U(x) + p_\beta(x) - p_\alpha(x))_{\triple \in \mathcal{R}}\Big)\, dx, \\
   \| (U, \bm{p} )\|_{\c,\Omega}^2 &:= \| \D U \|_{L^2(\Omega)}^2
               + \sum_{\alpha=0}^{S-1} \| p_\alpha - p_0 \|_{L^2(\Omega)}^2, \\
								\| (U, \bm{p} )\|_{\c,\mathbb{R}^d}^2 &:= \| \D U \|_{L^2(\mathbb{R}^d)}^2
               + \sum_{\alpha=0}^{S-1} \| p_\alpha - p_0 \|_{L^2(\mathbb{R}^d)}^2.
\end{align*}
We will evaluate $\mathcal{E}^\c_{\Omega}$ only for $(U, \bm{p}) \in C^1_{\rm
per}(\Omega) \times ( C_{\rm per}(\Omega) )^{S-1}$, where ${\rm per}$ denotes
periodic functions. For such fields $(U, \bm{p})$ we define the corresponding
atomistic fields
\begin{equation} \label{eq:c2at interpolation}
   u^\epsilon_\alpha(\xi) = U(\xi) +\epsilon p_\alpha(\xi) \quad
   \text{and} \quad
   \bm{u}^\epsilon = (u^\epsilon_\alpha)_{\alpha = 0}^{S-1}.
\end{equation}

The next result is a scaled variant of~\cite[Proposition 3.1]{koten2013},
proven by a straightforward Taylor expansion.

\begin{lemma}\label{cbEnergyEstimate}
   Let $U \in C^{3}(\Omega), p_\alpha \in C^{2}(\Omega)$, and
   $\bm{u}^\epsilon$ given by \eqref{eq:c2at interpolation}. Then there exists
   a constant $C$, independent of $\epsilon$, such that
   \begin{align*}
   \left|\mathcal{E}^\a_\epsilon(\bm{u}^\epsilon)
   - \mathcal{E}^\c_{\Omega}((U,\bm{p}))\right| \leq C \epsilon.
   \end{align*}
\end{lemma}

Arguing as in \cite[Lemma 3.2]{hudson2013}, Lemma~\ref{cbEnergyEstimate} implies
convergence of hessians.  The proof requires only minor adjustments.

% Next, we claim the atomistic and continuum Hessians converge for the reference
% state $\bm{y}^\epsilon(\xi) = (\mF\xi + \epsilon p_\alpha)_{\alpha = 0}^{S-1}$,
% which we write succinctly in our displacement-shift format as $(U,\bm{p}) = 0$.
% It is a restatement of~\cite[Lemma 3.2]{hudson2013} for multilattices,

\begin{lemma}\label{hessianConvergence}
%Suppose that $\bm{u}^\epsilon = (\mF + \epsilon p_\alpha)_{\alpha = 0}^{S-1}$ (so that $U(x) = \mF x$ and $p_\alpha(x) \equiv p_\alpha$) is an equilibrium of $\mathcal{E}^\a_\epsilon$.
Let $Z \in {\rm C}^3(\Omega), q_\alpha \in {\rm C}^2(\Omega)$ and
$\bm{z}^\epsilon$ define analogously to \eqref{eq:c2at interpolation}, then
\[
\<\delta^2 \mathcal{E}_\epsilon^\a(0)\bm{z}^\epsilon, \bm{z}^\epsilon\> - \<\delta^2 \mathcal{E}^\c_{\Omega}(0)(Z, \bm{q}), (Z, \bm{q})\> \to 0, \quad \mbox{as $\epsilon \to 0$.}
\]
\end{lemma}

We have now assembled the necessary prerequisites to prove that atomistic stability,
 Assumption~\ref{coercive}, implies stability of the Cauchy--Born model.

% We now have all of the necessary tools at hand to prove atomistic stability from
% Assumption~\ref{coercive} implies Cauchy--Born stability of the defect-free
% lattice, which is the multilattice version of~\cite[Lemma 5.2]{theil2012}

\begin{theorem}[Cauchy--Born Stability]\label{CBstability}
   Suppose $Z \in H^1_{\rm loc}(\mathbb{R}^d,\mathbb{R}^n), q_\alpha \in L^2(\mathbb{R}^d,\mathbb{R}^n)$
   with $\nabla Z, q_\alpha$ having compact support. Then there exists
   $\gamma_\c > 0$ such that
   \begin{equation} \label{eq:CBstability}
   \big\< \delta^2 \mathcal{E}^\c(0)(Z,\bm{q}),(Z,\bm{q}) \big\>
   \geq \gamma_\c \big\| (Z, \bm{q}) \big\|_{\c, \R^d}^2.
   % \big[\|\nabla Z\|_{L^2(\mathbb{R}^d)}^2 + \sum_{\alpha} \|q_\alpha\|_{L^2(\mathbb{R}^d)}^2\big].
   \end{equation}
\end{theorem}
\begin{proof}
   This proof largely follows the Bravais lattice case \cite{hudson2012};
   the main additional step is the correct choice of rescaling the shifts.

   Suppose $(Z, \bm{q})$ is supported in $B_{R/2}$.
   Rescaling $Z_R(x) := R^{-1} Z(Rx)$ and $q_{R,\alpha}(x) := q_\alpha(Rx)$, we obtain
   that $(Z_R, \bm{q}_R)$ has support contained in $B_{1/2}(0)$, while
   \begin{align*}
      \big\< \delta^2 \mathcal{E}^\c(0)(Z,\bm{q}),(Z,\bm{q}) \big\>
      &= R^d \big\< \delta^2 \mathcal{E}^\c(0)(Z_R,\bm{q}_R),(Z_R,\bm{q}_R) \big\>,
      \quad \text{and} \quad \\
      \| (Z,\bm{q}) \|_{\c, \R^d}^2 &= R^d \| (Z_R,\bm{q}_R) \|_{\c, \Omega}^2.
   \end{align*}
 In particular,
   stability
   for $(Z, \bm{q})$ implies stability for $(Z_R, \bm{q}_R)$ and vice-versa,
   that is, we drop the subscript $R$ and assume, without loss of generality,
   that $(Z, \bm{q})$ has support in $B_{1/2}$. Moreover, by density of smooth
   functions we may also assume that $(Z, \bm{q}) \in {\rm C}^3 \times ({\rm C}^2)^{S-1}$.

   We can now interpret $(Z,\bm{q})$ as periodic with respect to the domain $\Omega$ and,
   for $N \in \N, \epsilon := 1/N$, let $\bm{z}^\epsilon$ be the
   corresponding periodic atomistic test function defined via~\eqref{eq:c2at interpolation}. Then,
   Lemma~\ref{hessianConvergence} implies
   \begin{equation}\label{leopard1}
      \big|\<\delta^2 \mathcal{E}_\epsilon^\a(0) \bm{z}^{\epsilon},  \bm{z}^{\epsilon}\>
      - \<\delta^2 \mathcal{E}^\c_\Omega(0)(Z, \bm{q}), (Z, \bm{q})\>  \big| \to 0 \quad \mbox{as $N \to \infty$.}
   \end{equation}
   From standard finite element interpolation error estimates we can deduce that
   \begin{equation} \label{eq:cbstab-prf:norm-convergence}
      \| \bm{z}^\epsilon \|_{\a,\epsilon} \to \| (Z, \bm{q}) \|_{\c, \Omega}
      \quad \text{as } N \to \infty \quad (\epsilon \to 0).
   \end{equation}

   We now rescale $\bm{z}_N(\xi) := N\bm{z}^\epsilon(\xi/N)$ if $\xi/N \in \Omega$
   and $\bm{z}_N(\xi) = 0$ otherwise.
    Assumption~\ref{coercive} and norm equivalence, Lemma~\ref{norm_equiv}, then
   imply
   \begin{align*}
      0 < \gamma_\a'
      &\leq
         \frac{\< \delta^2 \mathcal{E}^\a(0)\bm{z}_N, \bm{z}_N \>}{
               \| \bm{z}_N \|_{\a_2}^2  }
      =
         \frac{ \< \delta^2 \mathcal{E}^\a_\epsilon(0) \bm{z}^\epsilon,
                  \bm{z}^\epsilon \> }{ \| \bm{z}^\epsilon \|_{\a, \epsilon}^2 } \\
      &\rightarrow
         \frac{\<\delta^2 \mathcal{E}^\c_\Omega(0)(Z, \bm{q}), (Z, \bm{q})\>}{
               \| (Z, \bm{q}) \|_{\c, \Omega}^2    } \qquad \text{as } N \to \infty,
   \end{align*}
   where we have used \eqref{leopard1} and \eqref{eq:cbstab-prf:norm-convergence}
   in the final line. Finally, for $(Z, \bm{q})$ supported in $\Omega$ we have
   \begin{equation*}
      \frac{\<\delta^2 \mathcal{E}^\c_\Omega(0)(Z, \bm{q}), (Z, \bm{q})\>}{
            \| (Z, \bm{q}) \|_{\c, \Omega}^2    }
      =
      \frac{\<\delta^2 \mathcal{E}^\c(0)(Z, \bm{q}), (Z, \bm{q})\>}{
            \| (Z, \bm{q}) \|_{\c, \mathbb{R}^d}^2    },
   \end{equation*}
   which completes the proof.
\end{proof}

\subsection{Lattice Green's function}

Having established the basic facts of the Fourier transform and Cauchy--Born model that we require, we now turn towards deriving the lattice Green's function to which we will apply these facts.  Applying the standard {\em continuous} Fourier transform on $\R^d$ to both sides of
\eqref{eq:CBstability} and applying the Plancherel theorem, we obtain
\begin{align*}
 \gamma_\c \Big( \int_{\mathbb{R}^d} 4\pi^2 |k|^2 |\hat{Z}(k)|^2 + \sum_{\alpha=0}^{S-1} |\hat{q}_\alpha(k)|^2\, dk \Big) \leq~& \big\<\delta^2 \mathcal{E}^\c(0) (Z, \bm{q}), (Z, \bm{q}) \big\> \\
=~& \int_{\mathbb{R}^d} \begin{pmatrix} \hat{Z}^* \\ \hat{\bm{q}}^* \end{pmatrix} \begin{pmatrix} J_{00}(k) &J_{0\bm{p}}(k) \\J_{0\bm{p}}^*(k) &J_{\bm{p}\bm{p}}(k) \end{pmatrix} \begin{pmatrix} \hat{Z} \\ \hat{\bm{q}}  \end{pmatrix} \, dk,
\end{align*}
where
\begin{align}
   \notag
   J_{00}(k) :=~& \sum_{\triple \in \mathcal{R}}\sum_{\tripleTau \in \mathcal{R}} 4\pi^2 (\tau \cdot k)
      V_{,\triple\tripleTau}(0) (\rho \cdot k) \\
   \notag
   [J_{0\bm{p}}(k)]_\beta :=~&  \
   \sum_{\rho \in \mathcal{R}_1}\sum_{\alpha = 0}^{S-1}\sum_{\tripleTau \in \mathcal{R}}
   -2\pi i (k \cdot \tau) \big[V_{,\triple\tripleTau}(0)
   - V_{,(\rho\beta\alpha)(\tau\gamma\delta)}(0)\big],  \\
   \notag
	 J_{\bm{p}0} =~& J_{0\bm{p}}^* \notag \\
   [J_{\bm{p}\bm{p}}(k)]_{\beta\delta} :=~& \sum_{\rho,\tau \in \mathcal{R}_1}\sum_{\alpha,\gamma = 0}^{S-1}
   \left[ V_{,\triple\tripleTau}(0) + V_{,(\rho\beta\alpha)(\tau\delta\gamma)}(0)
   - V_{,(\rho\beta\alpha)(\tau\gamma\delta)}(0) \right. \\
   \notag
   &\qquad\qquad \left. - V_{,(\rho\alpha\beta)(\tau\delta\gamma)}(0)\right], \\
   \notag
   J^{-1}(k) :=~& \begin{pmatrix} M^{-1} &-M^{-1}J_{0\bm{p}}J_{\bm{p}\bm{p}}^{-1} \\
   \notag
    -J_{\bm{p}\bm{p}}^{-1}J_{\bm{p}0}M^{-1} &
     J_{\bm{p}\bm{p}}^{-1}J_{\bm{p}0}M^{-1}J_{0\bm{p}}J_{\bm{p}\bm{p}}^{-1}
      + J_{\bm{p}\bm{p}}^{-1} \end{pmatrix}, \, \mbox{where}\\
    \label{eq:def-M-matrix}
   M :=~& J_{00} - J_{0\bm{p}} J_{\bm{p}\bm{p}}^{-1} J_{\bm{p}0}.
\end{align}

By taking the test pair with $\bm{q} = \bm{0}$, we see that this implies
\[
    \gamma_\c \Big( \int_{\mathbb{R}^d} 4 \pi^2 |k|^2 |\hat Z (k)|^2 \, dk \Big) \leq \int_{\mathbb{R}^d} \hat{Z}^* J_{00}(k) \hat{Z} \, dk,
\]
and in particular we obtain
\begin{equation*}
   J_{00}(k) \geq \gamma_\c 4\pi^2 |k|^2 I_{n \times n} \qquad
   \text{ for } k \in \mathbb{R}^d \setminus \{0\}.
\end{equation*}

In a similar fashion, by testing with pairs having $Z = 0$, we see that
\[
 \gamma_\c \Big( \int_{\mathbb{R}^d} \sum_{\alpha=0}^{S-1} |\hat{q}_\alpha(k)|^2 \, dk \Big) \leq \int_{\mathbb{R}^d} \hat{\bm{q}}^* J_{\bm{p}\bm{p}} \hat{\bm{q}} \, dk,
\]
where $J_{\bm{p}\bm{p}}$ is symmetric and independent of $k$, hence
\begin{equation*}
   J_{\bm{p}\bm{p}} \geq \gamma_\c I_{(S-1)n \times (S-1)n}.
\end{equation*}
% and so the eigenvalues, $\lambda_{\rm o}$, of $J_{\bm{p}\bm{p}}$ must be strictly positive.

Next, we note that $M = J_{00} - J_{0\bm{p}} J_{\bm{p}\bm{p}}^{-1} J_{\bm{p}0}$
is the Schur complement of $J_{\bm{p}\bm{p}}$ in $J$.  \cite[Theorem 5]{smith1992}
or~\cite[Corollary 2.3]{zhang2006schur} imply that the eigenvalues of $M$ interlace those
of $J$, so in particular we obtain that $M(k) \geq c |k|^2$ for some $c > 0$.
%  so we must also have the spectrum, $\mu_\a$, of $M$, satisfying
% $\mu_\a(k) \gtrsim |k|^2$.
Letting ${\rm A}_{ijkl} :=
\frac{\partial \bar{W}(\mG)}{\partial \mG_{ij}\partial \mG_{kl}}$ with $\mG$ defined as in Lemma~\ref{th:equil_shift_cor}, we obtain
\begin{equation}\label{claimant}
   \int_{\mathbb{R}^d} {\rm A}_{ijkl}Z_{i,j}Z_{k,l} \, dx =: \int_{\mathbb{R}^d} {\rm A}:\nabla Z:\nabla Z\, dx
   = \int_{\mathbb{R}^d} \hat{Z}^* M \hat{Z} \, dk
   \gtrsim \|\nabla Z\|^2_{\mathbb{R}^d}.
\end{equation}
The proof of \eqref{claimant}, presented in \S~\ref{sec:proof-claimant}, is a
tedious algebraic manipulation, the key observation being that
$\partial_{\bm{p}} \hat{W}((\mF,\bm{p})) = 0$, which we have proven holds in Lemma
\ref{th:equil_shift_cor} since we assume the reference configuration is in equilibrium.

It follows from \eqref{claimant} that ${\rm A}$ satisfies
the Legendre--Hadamard ellipticity condition. We can therefore apply~\cite[Equation
6.2.15]{Morrey} to obtain bounds on the Green's matrix for the linearized
continuum elasticity operator.

\begin{lemma} \label{th:CB-greens-fcn}
   Let $M$ be defined by \eqref{eq:def-M-matrix}. Then the Green's function,
   $\check{M}(x)$, for the differential operator
   ${\rm div}  \big({\rm A} \nabla \cdot \big)$ satisfies the decay rates
   \begin{equation}\label{horse0}
      |\nabla^j \check{M}(x)| \lesssim~ (1 + |x|)^{2-d-j}.
   \end{equation}
\end{lemma}

\section{Proof of Theorem~\ref{main_thm}}\label{prove}

To prove our main result, Theorem~\ref{main_thm}, we first linearize the
equilibrium equation \eqref{eq:equlibrium} about the ground state.
We then use the decay estimate \eqref{horse0} for the Cauchy--Born Green's
function to obtain a corresponding estimate for the atomistic Green's function.
This will then allow us to prove the decay rates for the
displacements and shifts stated in Theorem~\ref{main_thm}.

\subsection{Linearized Equation}\label{s:linearized}

The linearized equation that $\bm{u}^\infty$ satisfies is formed by linearizing
the defect-free energy $\mathcal{E}^\a_{\rm hom}$ about the reference state
$\bm{u} = \bm{0}$.  The
key point in this linearization is that the residual is quadratic in terms of
the defect solution.

% For each $\triple \in \mathcal{R}$ and each $\xi \in \mathcal{L}$, there exists $f_{\rho\alpha\beta}(\xi)$ with $\{\sum_{\triple \in \mathcal{R}} |f_{\triple}(\xi)|^2\}^{1/2} \lesssim |D\bm{u}^\infty(\xi)|^2$ such that
% \begin{equation}

\begin{theorem}\label{decay_cor1}
   There exists $f : \mathcal{L} \to (\R^n)^{\mathcal{R}}$ such that
   \begin{align}
      \label{eq:residual-bound}
      \<\delta^2 \mathcal{E}^\a_{\rm hom}(0)\bm{u}^\infty, \bm{v}\>
      &= \sum_{\xi \in \mathcal{L}}\sum_{\triple \in \mathcal{R}} f_{\triple}(\xi)\cdot D_{\triple}\bm{v}(\xi) =: \<f, D\bm{v}\>, \quad \forall \, \bm{v} \in \bm{\mathcal{U}}_0, \\
      \text{where} \quad
    |f(\xi)|_{\mathcal{R}} &\lesssim |D\bm{u}^\infty(\xi)|_{\mathcal{R}}^2
            \qquad \text{ for }|\xi| \geq \Rdef.
      \label{decay2}
   \end{align}
\end{theorem}
\begin{proof}
By~\eqref{ostrich1},  $\<\delta \mathcal{E}^\a_{\rm hom}(0), \bm{v}\> = 0$ for all  $\bm{v} \in \bm{\mathcal{U}}_0$. Hence
\begin{equation*}\label{emu1}
\begin{split}
   &\<\delta^2 \mathcal{E}^\a_{\rm hom}(0)\bm{u}^\infty, \bm{v}\> =~ \<\delta^2 \mathcal{E}^\a_{\rm hom}(0)\bm{u}^\infty, \bm{v}\> - \<\delta \mathcal{E}^\a_{\rm hom}(\bm{u}^{\infty}), \bm{v}\>  + \<\delta \mathcal{E}^\a_{\rm hom}(\bm{u}^{\infty}), \bm{v}\>  \\
&=~ \Big\{\<\delta^2 \mathcal{E}^\a_{\rm hom}(0)\bm{u}^\infty, \bm{v}\> + \<\delta \mathcal{E}^\a_{\rm hom}(0), \bm{v}\> - \<\delta \mathcal{E}^\a_{\rm hom}(\bm{u}^{\infty}), \bm{v}\> \Big\} + \<\delta \mathcal{E}^\a_{\rm hom}(\bm{u}^{\infty}), \bm{v}\> \\
&=~  - L_1[\bm{u}^\infty,\bm{v}] + \<\delta \mathcal{E}^\a_{\rm hom}(\bm{u}^{\infty}), \bm{v}\>,
\end{split}
\end{equation*}
where $L_1$ is a linearization residual of the form
\begin{equation}\label{emu2}
   L_1[\bm{u}^\infty,\bm{v}] := \sum_{\xi \in \mathcal{L}} \bigg\< \frac{1}{2}
      \int_0^1 \delta^3 V(tD\bm{u}^\infty(\xi))[(1-t)D\bm{u}^{\infty}(\xi),
                  (1-t)D\bm{u}^{\infty}(\xi)], D\bm{v}(\xi)\bigg\> \, dt.
\end{equation}

Next, note that $\<\delta\mathcal{E}^\a(\bm{u}^\infty),\bm{v}\> = 0$ for all test functions $\bm{v}$ since $\bm{u}^\infty$ is a critical point of $\mathcal{E}^a$, and recall that $V_\xi \equiv V$ for $|\xi| \geq \Rdef$.  Thus,
\begin{equation}\label{emu3}
\begin{split}
\<\delta \mathcal{E}^\a_{\rm hom}(\bm{u}^{\infty}), \bm{v}\> =~& \<\delta \mathcal{E}^\a_{\rm hom}(\bm{u}^{\infty}), \bm{v}\> -  \<\delta\mathcal{E}^\a(\bm{u}^\infty),\bm{v}\> \\
=~&  \sum_{\xi \in \mathcal{L}\cap B_{\Rdef}(0)} \Big\<
      \delta V(D\bm{u}^\infty(\xi)) - \delta V(D\bm{u}^\infty(\xi)),
      D\bm{v}(\xi) \Big\>
\end{split}
\end{equation}
Combining~\eqref{emu2} and~\eqref{emu3}, we define
\begin{align*}
&f_{\triple}(\xi) \\
&= \begin{cases}&\frac{1}{2}\int_0^1 \delta^3V(t\bm{u}^\infty(\xi))[(1-t)\bm{u}^{\infty}(\xi),(1-t)\bm{u}^{\infty}(\xi)]\, dt  \\
&\qquad +~ V_{,\triple}(D\bm{u}^\infty(\xi)) - V_{\xi, \triple}(D\bm{u}^\infty(\xi))  , \quad \mbox{if} \quad \xi \in B_{\Rdef}(0), \\
   &\frac{1}{2}\int_0^1 \delta^3V(t\bm{u}^\infty(\xi))[(1-t)\bm{u}^{\infty}(\xi),(1-t)\bm{u}^{\infty}(\xi)]\, dt,                                \quad \mbox{if} \quad \xi \notin B_{\Rdef}(0),
\end{cases}
\end{align*}
and note that $f_{\triple}(\xi)$ satisfies the desired bounds since $B_{\Rdef}(0)$ is finite and since the third derivative of $V$ is bounded by our assumptions on the site potential.
\end{proof}

\subsection{Dynamical Matrix and Green's Function}

We now construct a Green's function representation for the solution of the
linearized equation~\eqref{decay2} so we convert it to an
equation in Fourier space. We rewrite the left-hand side in real space in
terms of the displacements $U^\infty := u_0^\infty$ and $Z := v_0$ and the
shifts $p_\alpha^\infty := u_\alpha^\infty - u_0^\infty$ and $q_\alpha :=
v_\alpha - v_0$,
\begin{align*}
&\<\delta^2 \mathcal{E}^\a_{\rm hom}(0)\bm{u}^\infty, \bm{v}\> =~ \sum_{\xi \in \mathcal{L}}\sum_{\triple \in \mathcal{R}}\sum_{\tripleTau\in \mathcal{R}} [D_{\tripleTau}\bm{v}(\xi)]^T V_{,\triple\tripleTau}(0) [D_{\triple}\bm{u}^\infty(\xi)] \\
&\qquad =~ \sum_{\xi \in \mathcal{L}}\sum_{\triple\in \mathcal{R}}\sum_{\tripleTau\in \mathcal{R}} [v_\delta(\xi+\tau) - v_\gamma(\xi)]^T V_{,\triple\tripleTau}(0) [u_\beta^\infty(\xi+\rho) - u_\alpha^\infty(\xi)] \\
&\qquad =~ \sum_{\xi \in \mathcal{L}}\sum_{\triple\in \mathcal{R}}\sum_{\tripleTau\in \mathcal{R}} [Z(\xi+\tau) - Z(\xi) + q_\delta(\xi+\tau) - q_\gamma(\xi)]^T V_{,\triple\tripleTau}(0)  \\
         & \hspace{6cm} \big[U^\infty(\xi+\rho)  - U^\infty(\xi) + p_\beta^\infty(\xi + \rho) - p_\alpha^\infty(\xi) \big],
\end{align*}
and then use the Plancherel Theorem to obtain
\begin{equation}\label{kiwi1}
\begin{split}
&\<\delta^2 \mathcal{E}^\a_{\rm hom}(0)\bm{u}^\infty, \bm{v}\> \\
&=~ \int_{\mathcal{B}} \sum_{\triple\in \mathcal{R}}\sum_{\tripleTau\in \mathcal{R}}
   \big[(e^{2\pi i k\cdot \tau}-1)\hat{Z}(k) + e^{2\pi i k\cdot \tau}\hat{q}_\delta(k) - \hat{q}_\gamma(k)\big]^*
      V_{,\triple\tripleTau}(0) \\
      & \hspace{6cm} \big[(e^{2\pi i k\cdot \rho}-1)\hat{U}^\infty(k)
            + e^{2\pi i k\cdot \rho}\hat{p}_\beta^\infty(k) - \hat{p}_\alpha^\infty(k)\big] \\
&=~ \int_{\mathcal{B}} \sum_{\triple\in \mathcal{R}}\sum_{\tripleTau\in \mathcal{R}} [(e^{2\pi i k\cdot \tau}-1)\hat{Z}(k)]^*V_{,\triple\tripleTau}(0)[(e^{2\pi i k\cdot \rho}-1)\hat{U}^\infty(k)] \\
&\quad+~ \int_{\mathcal{B}} \sum_{\triple\in \mathcal{R}}\sum_{\tripleTau\in \mathcal{R}} [(e^{2\pi i k\cdot \tau}-1)\hat{Z}(k)]^*V_{,\triple\tripleTau}(0)[e^{2\pi i k\cdot \rho}\hat{p}_\beta^\infty(k) - \hat{p}_\alpha^\infty(k)] \\
&\quad+~ \int_{\mathcal{B}} \sum_{\triple\in \mathcal{R}}\sum_{\tripleTau\in \mathcal{R}} [e^{2\pi i k\cdot \tau}\hat{q}_\delta(k) - \hat{q}_\gamma(k)]^*V_{,\triple\tripleTau}(0)[(e^{2\pi i k\cdot \rho}-1)\hat{U}^\infty(k)] \\
&\quad+~ \int_{\mathcal{B}} \sum_{\triple\in \mathcal{R}}\sum_{\tripleTau\in \mathcal{R}} [e^{2\pi i k\cdot \tau}\hat{q}_\delta(k) - \hat{q}_\gamma(k)]^*V_{,\triple\tripleTau}(0)[e^{2\pi i k\cdot \rho}\hat{p}_\beta^\infty(k) - \hat{p}_\alpha^\infty(k)].
%  \\
% &=: T_1 + T_2 + T_3 + T_4.
\end{split}
\end{equation}

In analogy to the Cauchy--Born Hessian, we now define
\begin{align*}
H_{00}(k) :=~& \sum_{\triple\in \mathcal{R}}\sum_{\tripleTau\in \mathcal{R}} (e^{-2\pi i k\cdot \tau}-1)V_{,\triple\tripleTau}(0)(e^{2\pi i k\cdot \rho}-1), \\
[H_{0\bm{p}}(k)]_\beta :=~&  \sum_{\rho\in\mathcal{R}_1}\sum_{\alpha = 0}^{S-1}\sum_{\tripleTau\in \mathcal{R}} [(e^{-2\pi i k\cdot \tau}-1)V_{,(\rho\alpha\beta)(\tau\gamma\delta)}(0)(e^{2\pi i k\cdot \rho}) - (e^{-2\pi i k\cdot \tau}-1)V_{,(\rho\beta\alpha)(\tau\gamma\delta)}(0)],  \\
[H_{\bm{p}0}(k)]_\delta :=~& \sum_{\tau\in\mathcal{R}_1}\sum_{\gamma = 0}^{S-1}\sum_{\triple\in \mathcal{R}} [(e^{-2\pi i k\cdot \tau}) V_{,(\rho\alpha\beta)(\tau\gamma\delta)}(0)(e^{2\pi i k\cdot \rho} - 1) -   V_{,(\rho\alpha\beta)(\tau\delta\gamma)}(0)(e^{2\pi i k\cdot \rho} - 1)], \\
[H_{\bm{p}\bm{p}}(k)]_{\beta\delta} :=~& \sum_{\rho,\tau \in \mathcal{R}_1}\sum_{\alpha,\gamma = 0}^{S-1}\left[e^{-2\pi i k\cdot \tau} V_{,(\rho\alpha\beta)(\tau\gamma\delta)}(0) e^{2\pi i k\cdot \rho} + V_{,(\rho\beta\alpha)(\tau\delta\gamma)}(0) - e^{-2\pi i k\cdot \tau}V_{,(\rho\beta\alpha)(\tau\gamma\delta)}(0) \right. \\
&\qquad\qquad \left. - e^{2\pi i k \cdot \rho} V_{,(\rho\alpha\beta)(\tau\delta\gamma)}(0)\right],
\end{align*}
and note that the matrix
\begin{equation}\label{dynam_matrix}
   H(k) := \begin{bmatrix} H_{00}(k) & H_{0\bm{p}}(k) \\ H_{\bm{p}0}(k) & H_{\bm{p}\bm{p}}(k)\end{bmatrix},
\end{equation}
known as the {\em dynamical matrix}~\cite{wallace1998},
is Hermitian due to $V_{,\triple\tripleTau}^{ij}(0) = V_{,\tripleTau\triple}^{ji}(0)$.

We may now rewrite~\eqref{kiwi1} succinctly as
\begin{equation}\label{decay_egg}
\begin{split}
\<\delta^2 \mathcal{E}^\a_{\rm hom}(0)\bm{u}^\infty, \bm{v}\> =~& \int_{\mathcal{B}} \begin{bmatrix} \hat{Z}(k) \\ \hat{\bm{q}}(k)\end{bmatrix}^* H(k) \begin{bmatrix} \hat{U}^\infty(k) \\ \hat{\bm{p}}^\infty(k) \end{bmatrix}\, dk.
\end{split}
\end{equation}

In order to give Assumption~\ref{coercive} an interpretation in terms of $H$, we introduce a third norm $\| \cdot \|_{\a_3}$ defined for $\bm{v} \equiv (Z,\bm{q})$ by
\[
\|\bm{v}\|_{\a_3}^2 = \|(Z, \bm{q})\|_{\a_3}^2 := \|2\pi|k|\hat{Z}\|_{L^2(\mathcal{B})}^2 + \sum_{\alpha = 1}^{S-1} \|\hat{p}_\alpha\|^2_{L^2(\mathcal{B})}.
\]
We show in \S~\ref{norm_equiv} that $\|\cdot\|_{\a_3}$ is equivalent to $\|\cdot \|_{\a_2}$ (and hence $\|\cdot\|_{\a_1}$). We then use Assumption~\ref{coercive} and~\eqref{decay_egg} to produce
\begin{equation}\label{kiwi3}
   \|(Z, \bm{q})\|^2_{\a_3} \lesssim \<\delta^2 \mathcal{E}^\a_{\rm hom}(0)\bm{v}, \bm{v}\>
= \int_{\mathcal{B}} \begin{bmatrix} \hat{Z}(k) \\ \hat{\bm{q}}(k)\end{bmatrix}^* H(k) \begin{bmatrix} \hat{Z}(k) \\ \hat{\bm{q}}(k) \end{bmatrix}\, dk.
\end{equation}
If $\bm{q} = \bm{0}$, then~\eqref{kiwi3} translates to
\begin{equation}\label{rhea1}
4\pi^2 \int_{\mathcal{B}} |k|^2 |\hat{Z}(k)|^2\, dk \lesssim~ \int_{\mathcal{B}} \hat{Z}(k)^* H_{00}(k) \hat{Z}(k)\, dk,
\end{equation}
while if $Z = 0$,~\eqref{kiwi3} implies
\begin{equation}\label{rhea2}
\int_{\mathcal{B}} |\hat{q}(k)|^2\, dk \lesssim~ \int_{\mathcal{B}} \hat{\bm{q}}(k)^* H_{\bm{p}\bm{p}}(k) \hat{\bm{q}}(k)\, dk.
\end{equation}
As the inequalities~\eqref{rhea1} and~\eqref{rhea2} are valid for all test functions, it follows that the spectra $\omega_0(k)$ of $H_{00}(k)$ and $\omega_{\bm{p}}(k)$ of $H_{\bm{p}\bm{p}}(k)$ satisfy the bounds
\begin{equation}\label{spectra}
\begin{split}
|k|^2 \lesssim~& \omega_0(k) \\
1 \lesssim~& \omega_{\bm{p}}(k),
\end{split}
\end{equation}
and in particular that $H_{00}(k)$ and $H_{\bm{p}\bm{p}}(k)$ are positive definite
for $k \neq 0$.

\begin{remark}\label{relation}
Since $H_{00}$ and $H_{\bm{p}\bm{p}}$ are principal submatrices of $H$, the Cauchy Interlacing Theorem, the spectral estimates~\eqref{spectra}, and Assumption~\ref{coercive} imply that there exist three positive eigenvalues, $\lambda^1_{\rm a}(k), \lambda^2_{\a}(k), \lambda^3_{\a}(k)$, of $H$ which satisfy
\begin{equation}\label{acoustic}
k^2 \lesssim~ \lambda^i_{\a}(k) \lesssim~ k^2
\end{equation}
and $S\cdot n - 3$ positive eigenvalues, $\lambda^i_{\rm o}$, of $H$ which satisfy
\begin{equation}\label{optical}
\lambda^i_{\rm o} \gtrsim 1.
\end{equation}
These are precisely the (squares of) frequencies  associated with the
\textit{acoustic} ($\lambda_\a^i$) and \textit{optical} ($\lambda_{\rm o}^i$)
phonon branches of the crystal~\cite{wallace1998}.  Comparing
Assumption~\ref{coercive} to~\cite[Assumption A]{weinan2007cauchy}, it thus
follows that Assumption~\ref{coercive} implies the bounds on the acoustic and
optical phonon frequencies stated in~\cite[Assumption A]{weinan2007cauchy}.
Moreover, using the norm equivalence between $\|\cdot \|_{\a_1}, \| \cdot \|_{\a_2}$, and $\|\cdot \|_{\a_3}$ along with~\cite[Section 6]{weinan2007cauchy}, the same assumptions on the acoustical
and optical frequencies can be used to show Assumption~\ref{coercive} is
satisfied so that the two assumptions are in fact equivalent.
\end{remark}

\medskip

Returning to~\eqref{decay2}, the right-hand side in Fourier space becomes
\begin{equation*}\label{decay_egg2}
\begin{split}
\<f, Dv\> &= \sum_{\triple\in \mathcal{R}}\sum_{\xi \in \mathcal{L}} f_{\triple}(\xi)\cdot D_{\triple}v(\xi) \\
&= \sum_{\triple\in \mathcal{R}}\int_{\mathcal{B}}\left[(e^{2\pi ik\cdot \rho}-1)\hat{Z}(k) + e^{2\pi i k\cdot \rho}\hat{q}_\beta(k) - \hat{q}_\alpha(k)\right]^*\hat{f}_{\rho\alpha\beta}(k) \\
%=~& \sum_{\rho\alpha\beta}\int_{\mathcal{B}}[(e^{2\pi ik\cdot \rho}-1)\hat{V}(k)]^*\hat{f}_{\rho\alpha\beta}(k) + \sum_{\rho\alpha\beta}\int_{\mathcal{B}} [\hat{v}_\beta(k)- \hat{v}_\alpha(k)]^*\hat{f}_{\rho\alpha\beta}(k) \\
&= \sum_{\triple\in \mathcal{R}}\int_{\mathcal{B}}[(e^{2\pi ik\cdot \rho}-1)\hat{Z}(k)]^*\hat{f}_{\triple}(k) + \sum_{\triple}\int_{\mathcal{B}} [e^{2\pi i k\cdot \rho}\hat{q}_\beta(k)- \hat{q}_\alpha(k)]^*\hat{f}_{\triple}(k) \\
&=: \int_{\mathcal{B}} \begin{bmatrix} \hat{Z}(k) \\ \hat{\bm{q}}(k) \end{bmatrix}^*\begin{bmatrix} F(k) \\ \bm{g}(k) \end{bmatrix},
\end{split}
\end{equation*}
where
\begin{equation}\label{residual}
\begin{split}
F(k) =~&  \sum_{\triple \in \mathcal{R}} (e^{-2\pi i k\cdot \rho}-1)\hat{f}_{\triple}(k),\\
g_\eta(k) =~& \sum_{(\rho\alpha\eta) \in \mathcal{R}} e^{-2\pi ik\cdot \rho}\hat{f}_{(\rho\alpha\eta)} - \sum_{(\rho\eta\beta) \in \mathcal{R}}\hat{f}_{(\rho\eta\beta)}.
\end{split}
\end{equation}
In summary, we have shown the following result.

\begin{theorem}\label{fourier_equation}
   Let $\bm{u}^\infty = (U^\infty, \bm{p}^\infty)$ be as in
   Assumption~\ref{coercive}. With $H(k), F(k)$, and $\bm{g}(k)$ as defined
   in \eqref{dynam_matrix} and \eqref{residual}, $(U^\infty, \bm{p}^\infty)$
   satisfies the linear system
   \begin{equation}\label{lin_equation}
      H(k)\begin{bmatrix}
            \hat{U}^\infty(k) \\
            \hat{\bm{p}}^\infty(k)
      \end{bmatrix} = \begin{bmatrix} F(k) \\ \bm{g}(k) \end{bmatrix}.
\end{equation}
\end{theorem}

Invertibility of $H(k)$ (except at $k = 0$) follows from~\eqref{spectra}
after using either the Schur complement, $Q := H_{00} - H_{0\bm{p}}H_{\bm{p}\bm{p}}^{-1}H_{\bm{p}0}$, of $H_{\bm{p}\bm{p}}$ in
$H$, or the Schur complement, $P := H_{\bm{p}\bm{p}} - H_{\bm{p}0}H_{00}^{-1}H_{0\bm{p}}$,
of $H_{00}$ in $H$ to write the inverse of $H$ as either (c.f.~\cite{zhang2006schur})
\begin{equation}\label{goose1}
\begin{split}
H^{-1}(k) =~& \begin{pmatrix} Q^{-1} &-Q^{-1}H_{0\bm{p}}H_{\bm{p}\bm{p}}^{-1} \\ -H_{\bm{p}\bm{p}}^{-1}H_{\bm{p}0}Q^{-1} & H_{\bm{p}\bm{p}}^{-1}H_{\bm{p}0}Q^{-1}H_{0\bm{p}}H_{\bm{p}\bm{p}}^{-1} + H_{\bm{p}\bm{p}}^{-1} \end{pmatrix}, \quad \text{or}  \\[2mm]
H^{-1}(k) =~& \begin{pmatrix} H^{-1}_{00} + H_{00}^{-1}H_{0\bm{p}}P^{-1}H_{\bm{p}0}H_{00}^{-1} &-H_{00}^{-1}H_{0\bm{p}}P^{-1} \\ -P^{-1}H_{\bm{p}0}H_{00}^{-1} & P^{-1} \end{pmatrix}.
\end{split}
\end{equation}
Setting $\mathcal{G} := (H^{-1})^{\vee}$ as the atomistic Green's function allows us to
write $U^\infty $ and $\bm{p}^\infty$ as a convolution
\begin{align*}
\begin{pmatrix} U^\infty \\ \bm{p}^\infty \end{pmatrix} = \mathcal{G} \ast \begin{bmatrix} \check{F} \\ \check{\bm{g}} \end{bmatrix},
\end{align*}
or, writing out the individual blocks,
\begin{equation}\label{rhea3}
\begin{split}
   U^\infty(\xi)
   =~& \big[Q^{-1}(k)\big]^{\vee} * \check{F} (\xi)
   + \big[-Q^{-1}H_{0\bm{p}}H_{\bm{p}\bm{p}}^{-1}\big]^{\vee} \ast \check{\bm{g}}(\xi) \\
   \bm{p}^\infty(\xi)
   =~& \big[(-Q^{-1}H_{0\bm{p}}H_{\bm{p}\bm{p}}^{-1})^*\big]^{\vee} \ast \check{F}(\xi)
   + \big[H_{\bm{p}\bm{p}}^{-1}H_{\bm{p}0}Q^{-1}H_{0\bm{p}}H_{\bm{p}\bm{p}}^{-1} + H_{\bm{p}\bm{p}}^{-1}\big]^{\vee}*\check{\bm{g}}(\xi).
%p_\alpha =~& \check{A_\alpha(k)^*} * \check{f}_0(k) + \left[\check{[\bm{A}_{\bm{p}\bm{p}}^{-1}(k)]} \check{f}_{\bm{p}}(k) \right]_\alpha.
\end{split}
\end{equation}

\subsection{Decay of the Green's Function}

The utility of the expression~\eqref{rhea3} comes from the fact that we can
estimate the decay of each of the matrix blocks involved in this formula by
comparing them to corresponding blocks in the Cauchy--Born Green's matrix
and employing the estimates of Lemma~\ref{lem:space1A}.

\begin{theorem}\label{hess_thm}
   Let $\bm{\rho} \in (\mathcal{R}_1)^t$, $t \geq 0$ and
   $|\bm{\rho}| := t$, then
\begin{align}
\left|D_{\bm{\rho}}[Q^{-1}(k)]^{\vee}(\xi)\right| \lesssim~& (1+|\xi|)^{-d-|\bm{\rho}|+2} \qquad |\bm{\rho}| \geq 1, \label{hessianDecay1} \\
\left|D_{\bm{\rho}}[-Q^{-1}H_{0\bm{p}}H_{\bm{p}\bm{p}}^{-1}]^{\vee}(\xi)\right| \lesssim~& (1+|\xi|)^{-d-|\bm{\rho}|+1} \qquad |\bm{\rho}| \geq 0, \label{hessianDecay2}\\
\left|D_{\bm{\rho}}[H_{\bm{p}\bm{p}}^{-1}H_{\bm{p}0}Q^{-1}H_{0\bm{p}}H_{\bm{p}\bm{p}}^{-1}]^{\vee}\right| + \left|D_{\bm{\rho}}[H_{\bm{p}\bm{p}}^{-1}]^{\vee}\right| \lesssim~& (1+|\xi|)^{-d-|\bm{\rho}|} \qquad \quad  |\bm{\rho}| \geq 0. \label{hessianDecay3}
\end{align}
\end{theorem}

We prove each of the three estimates in Theorem~\ref{hess_thm} individually.
Throughout these proofs, if $\gamma \in \mathbb{N}_0^d$ is a multi-index, then $|\gamma| :=
\sum_{i = 1}^d \gamma_i$ denotes its length and $\partial_\gamma :=
\partial_{k_1}^{\gamma_1} \cdots \partial_{k_d}^{\gamma_d}$ the associated
partial differential operator.

\begin{proof}[Proof of~\eqref{hessianDecay1} of Theorem~\ref{hess_thm}]
   Let $\hat{\eta} \in C^\infty$ with ${\rm supp}(\hat{\eta}) \subset\subset
   \mathcal{B}$, then arguing similarly as in the proof of
    \cite[Lemma 6.2]{Ehrlacher2013}, we estimate
   \begin{equation}\label{horse1}
   \begin{split}
   |D_{\bm{\rho}} (Q^{-1})^{\vee}(\xi)| \leq~& |\eta* D_{\bm{\rho}} (M^{-1})^{\vee}(\xi)| +  |D_{\bm{\rho}}\big( (\eta * M^{-1})^{\vee}(\xi) - (Q^{-1})^{\vee}(\xi)\big)| \\
   \lesssim~& (1 + |\xi|)^{2-d-|\rho|} +  |D_{\bm{\rho}}\big( (\eta * M^{-1})^{\vee}(\xi) - (Q^{-1})^{\vee}(\xi)\big)|,
   \end{split}
   \end{equation}
   where we have used the estimate in~\eqref{horse0}. Next, we assume that
   $\hat{\eta} = 1$ on $B_{\epsilon}(0)$ for some $\epsilon > 0$, then
   for each multi-index $\gamma \in \mathbb{N}_0^d$, and for $k \in B_\epsilon$,
   \begin{equation*}
      \partial_\gamma (\hat{\eta}M^{-1} - Q^{-1})
      = \partial_\gamma (M^{-1}(Q - M)Q^{-1})
   \end{equation*}
   %
   % \begin{equation*}\label{horse3}
   % \begin{split}
   % \int_{\mathcal{B}} |\partial_\gamma (\hat{\eta}M^{-1} - Q^{-1})|\, dk
   % &=~ \int_{B_\epsilon(0)} |\partial_\gamma (M^{-1} - Q^{-1}) | \, dk + \int_{\mathcal{B} \backslash B_\epsilon(0)} |\partial_\gamma (\hat{\eta}M^{-1} - Q^{-1}) | \, dk \\
   % &\hspace{-1cm} =~ \int_{B_\epsilon(0)} |\partial_\gamma (M^{-1}(Q - M)Q^{-1}) | \, dk + \int_{\mathcal{B} \backslash B_\epsilon(0)} |\partial_\gamma (\hat{\eta} M^{-1} - Q^{-1}) | \, dk.
   % \end{split}
   % \end{equation*}
   From the expressions for $Q$ and $M$, it is clear that $\partial_\gamma(Q- M) =
   \mathcal{O}(k^{3-|\gamma|})$ and both $\partial_\gamma M^{-1} = \mathcal{O}(k^{-2-|\gamma|})$
   and $\partial_\gamma Q^{-1} = \mathcal{O}(k^{-2-|\gamma|})$ so $\partial_{\gamma}(M^{-1}(Q -
   M)Q^{-1}) = \mathcal{O}(k^{-1-|\gamma|})$.
   Outside of $B_\epsilon$, $\partial_\gamma (\hat{\eta}M^{-1} - Q^{-1})$
   is bounded. Hence, it follows from Corollary~\ref{lem:space1A}
   that
   \[
      |D_{\bm{\rho}}\big( (\eta * M^{-1})^{\vee}(\xi) - (Q^{-1})^{\vee}(\xi)\big)| \lesssim~ (1 + |\xi|)^{2-d-|\bm{\rho}|},
   \]
   which, combined with \eqref{horse1}, completes the proof.
\end{proof}

\begin{proof}[Proof of~\eqref{hessianDecay2} of Theorem~\ref{hess_thm}]
   Recall the definition of $H_{0\bm{p}}$ and $J_{0\bm{p}}$ as
\begin{align*}
[H_{0\bm{p}}(k)]_\beta :=~&  \sum_{\rho \in \mathcal{R}_1}\sum_{\alpha = 0}^{S-1}\sum_{\tripleTau\in \mathcal{R}} [(e^{-2\pi i k\cdot \tau}-1)V_{,(\rho\alpha\beta)(\tau\gamma\delta)}(0)(e^{2\pi i k\cdot \rho}) - (e^{-2\pi i k\cdot \tau}-1)V_{,(\rho\beta\alpha)(\tau\gamma\delta)}(0)], \\
[J_{0\bm{p}}(k)]_\beta :=~& \sum_{\rho \in \mathcal{R}_1}\sum_{\alpha = 0}^{S-1}\sum_{\tripleTau\in \mathcal{R}} (-2\pi i k\cdot \tau)V_{,(\rho\alpha\beta)(\tau\gamma\delta)}(0) + (-2\pi i k\cdot \tau)V_{,(\rho\beta\alpha)(\tau\gamma\delta)}(0)].
\end{align*}
To avoid double-subscripts we will write $J_{0\bm{p}}^\beta(k) := [J_{0\bm{p}}(k)]_\beta$ to mean the $\beta$ block of $J_{0\bm{p}}$ and $[H_{\bm{p}\bm{p}}^{-1}]^{\beta\gamma}$ to denote the $\beta\chi$ block of $H_{\bm{p}\bm{p}}^{-1}$.

As before let $\hat{\eta} \in C^\infty(\mathcal{B})$
 with ${\rm supp}(\hat{\eta}) \subset\subset
\mathcal{B}$ and $\hat{\eta} = 1$ in a ball $B_\epsilon$ contained in $\mathcal{B}$.  Our first step will be to show that
\begin{equation}\label{lemming}
   \big|D_{\bm{\rho}}\big(\eta * \big[M^{-1}H_{0\bm{p}} H_{\bm{p}\bm{p}}^{-1}\big]^{\vee}\,\big)\big| \lesssim |x|^{1-d-|\bm{\rho}|},
\end{equation}
after which we will estimate the difference $|\partial_\gamma\big(\hat{\eta} (M^{-1}(H_{0\bm{p}}-J_{0\bm{p}})P^{-1})\big)|$.  From properties of Fourier transforms
\begin{equation}\label{swan1}
\big[[M^{-1}]_{ij}{[J_{0\bm{p}}^\beta]}_{jm}[H_{\bm{p}\bm{p}}^{-1}]_{mn}^{\beta\chi}\big]^{\vee} =~ \big[[M^{-1}]_{ij}{[J_{0\bm{p}}^\beta]}_{jm}\big]^{\vee}* \big[[H_{\bm{p}\bm{p}}^{-1}]_{mn}^{\beta\chi}\big]^{\vee}.
\end{equation}
We take the \textit{full-space} inverse Fourier transform to find
\begin{equation*}\label{swan2}
\big[[M^{-1}]_{ij}{[J_{0\bm{p}}^\beta]}_{jm}\big]^{\vee}(x) = \sum_{s=1}^d\frac{\partial}{\partial x_s}\bigg[\sum_{\tau\delta\gamma}\sum_{\alpha = 0}^{S-1}\tau_s[V_{,(\rho\alpha\beta)(\tau\gamma\delta)}(0) + V_{,(\rho\alpha\beta)(\tau\delta\gamma)}(0)]_{jm} [M^{-1}]_{ij}^{\vee}(x)\bigg],
\end{equation*}
and then use~\cite[Equation 6.2.15]{Morrey} to deduce that
\begin{equation}\label{swan3}
\big|D_{\bm{\rho}}\big([M^{-1}]_{ij}{[J_{0\bm{p}}^\beta]}_{jm}\big)^{\vee}(x)\big| \lesssim |x|^{1-d-|\bm{\rho}|}.
\end{equation}

Furthermore, from the equality~\eqref{swan1} and the fact that convolution is commutative and associative,
\begin{equation}\label{swan4}
\eta * ([M^{-1}]_{ij}{[J_{0\bm{p}}^\beta]}_{jm})^{\vee}* \big[[H_{\bm{p}\bm{p}}^{-1}]_{mn}^{\beta\chi}\big]^{\vee} = ([M^{-1}]_{ij}{[J_{0\bm{p}}^\beta]}_{jm})^{\vee} * (\eta * \big[[H_{\bm{p}\bm{p}}^{-1}]_{mn}^{\beta\chi}\big]^{\vee}).
\end{equation}
Finally, the convolution on the right-hand side of~\eqref{swan4} will decay at
the slower of the two rates involved in the convolution.  Because $\eta *
\big[[H_{\bm{p}\bm{p}}^{-1}]_{mn}^{\beta\chi}\big]^{\vee}$ is the inverse Fourier transform of a
smooth function with compact support, it follows that this function is of
Schwartz class so decays faster than any polynomial.  Since finite differences
commute with convolutions, combining~\eqref{swan4} with~\eqref{swan3} then
yields~\eqref{lemming}.

In the following we will employ the estimates
\begin{equation}\label{lem:four}
\begin{split}
\partial_\gamma\big(\hat{\eta} (M^{-1}(J_{0\bm{p}}-H_{0\bm{p}})H_{\bm{p}\bm{p}}^{-1})\big) =~& \mathcal{O}(k^{-|\gamma|}), \\
\partial_\gamma\big(\hat{\eta} (Q^{-1}-M^{-1})H_{0\bm{p}}H_{\bm{p}\bm{p}}^{-1}\big) =~& \mathcal{O}(k^{-|\gamma|}),
\end{split}
\end{equation}
which can be readily established.

We now split
\begin{equation}\label{igloo1}
\begin{split}
|D_{\bm{\rho}}\big(Q^{-1}H_{0\bm{p}}H_{\bm{p}\bm{p}}^{-1}\big)^{\vee}(\xi)|
&\leq~ \left|D_{\bm{\rho}}\big(Q^{-1}H_{0\bm{p}}H_{\bm{p}\bm{p}}^{-1}\big)^{\vee}(\xi) - D_{\bm{\rho}}\big(\eta *(M^{-1}J_{0\bm{p}}H_{\bm{p}\bm{p}}^{-1})^{\vee}\big)(\xi)\right|  \\
&\qquad + \left| D_{\bm{\rho}}\big(\eta *(M^{-1}J_{0\bm{p}}H_{\bm{p}\bm{p}}^{-1})\big)^{\vee}(\xi)\right|.
\end{split}
\end{equation}
We already know the decay of the second term from~\eqref{lemming},
%to be
% \begin{equation}\label{igloo0}
% \left| D_{\bm{\rho}}\big(\eta *(M^{-1}J_{0\bm{p}}H_{\bm{p}\bm{p}}^{-1})\big)^{\vee}(\xi)\right| \lesssim |\xi|^{1-d-|\bm{\rho}|}.
% \end{equation}
hence we focus on the first term on the right-hand side of \eqref{igloo1}.
We take its Fourier transform and then a derivative of order $\gamma \in \mathbb{N}_0^d$ with the goal being to apply Corollary~\ref{lem:space1A}:
\begin{equation}\label{igloo2}
\begin{split}
&\partial_\gamma\Big(\big(Q^{-1}H_{0\bm{p}}H_{\bm{p}\bm{p}}^{-1}\big) - \hat{\eta}(M^{-1}J_{0\bm{p}}H_{\bm{p}\bm{p}}^{-1})\Big) \\
&\qquad =~ \partial_\gamma\Big(\big(Q^{-1}H_{0\bm{p}}H_{\bm{p}\bm{p}}^{-1}\big) - \hat{\eta}M^{-1}H_{0\bm{p}}H_{\bm{p}\bm{p}}^{-1} + \hat{\eta}M^{-1}H_{0\bm{p}}H_{\bm{p}\bm{p}}^{-1} -\hat{\eta}(M^{-1}J_{0\bm{p}}H_{\bm{p}\bm{p}}^{-1}) \Big) \\
&\qquad =~ \partial_\gamma\Big(\big(Q^{-1}H_{0\bm{p}}H_{\bm{p}\bm{p}}^{-1}\big) - \hat{\eta}M^{-1}H_{0\bm{p}}H_{\bm{p}\bm{p}}^{-1}\Big) + \partial_\gamma\Big(\hat{\eta}M^{-1}\big[H_{0\bm{p}} - J_{0\bm{p}}\big]H_{\bm{p}\bm{p}}^{-1})\Big).
\end{split}
\end{equation}
Combining \eqref{lem:four} and the properties of $\hat{\eta}$ we obtain
% , we have that
% \begin{equation}\label{igloo3}
% \partial_\gamma\Big(\hat{\eta}M^{-1}H_{0\bm{p}}H_{\bm{p}\bm{p}}^{-1} -\hat{\eta}(M^{-1}J_{0\bm{p}}H_{\bm{p}\bm{p}}^{-1})\Big) = \mathcal{O}(k^{-|\gamma|}),
% \end{equation}
% Using the properties of $\hat{\eta}$,
\begin{equation}\label{igloo4}
\partial_\gamma\left(\big(Q^{-1}H_{0\bm{p}}H_{\bm{p}\bm{p}}^{-1}\big) - \hat{\eta}M^{-1}H_{0\bm{p}}H_{\bm{p}\bm{p}}^{-1}\right) =~ \mathcal{O}(k^{-|\gamma|}).
\end{equation}
Applying Corollary~\ref{lem:space1A} to the estimates in~\eqref{igloo2}, \eqref{lem:four} and~\eqref{igloo4} yields
\begin{equation}\label{lem:five}
\begin{split}
\Big|D_{\bm{\rho}}\Big(\hat{\eta} \big(M^{-1}(J_{0\bm{p}}-H_{0\bm{p}})H_{\bm{p}\bm{p}}^{-1}\big)\Big)^{\vee}(\xi)\Big| \lesssim~ (1 + |\xi|)^{1-d-|\bm{\rho}|} \\
\Big|D_{\bm{\rho}}\Big( \big(Q^{-1}H_{0\bm{p}}H_{\bm{p}\bm{p}}^{-1}-\hat{\eta}M^{-1}H_{0\bm{p}}H_{\bm{p}\bm{p}}^{-1}\big)\Big)^{\vee}(\xi)\Big| \lesssim~ (1 + |\xi|)^{{1-d-|\bm{\rho}|}}.
\end{split}
\end{equation}
Combining the estimates in~\eqref{lem:five} and~\eqref{lemming} and using them in the decomposition~\eqref{igloo1} gives the desired
decay estimate \eqref{hessianDecay2}.
% \[
% |D_{\bm{\rho}}\big(Q^{-1}H_{0\bm{p}}H_{\bm{p}\bm{p}}^{-1}\big)^{\vee}(\xi)| \lesssim~ (1 + |\xi|)^{1-d-|\bm{\rho}|}. \qedhere
% \]
\end{proof}

\begin{proof}[Proof of~\eqref{hessianDecay3} of Theorem~\ref{hess_thm}]
   % First, we prove a related estimate on the continuum Green's matrix,
% \[
% \left|D_{\bm{\rho}}(H_{\bm{p}\bm{p}}^{-1}H_{\bm{p}0}Q^{-1}H_{0\bm{p}}H_{\bm{p}\bm{p}}^{-1})^{\vee}\right| + \left|D_{\bm{\rho}}(H_{\bm{p}\bm{p}}^{-1})^{\vee}\right| \lesssim~ (1+|\xi|)^{-d}.
% \]
To prove the second part of the estimate,
\[
\left|D_{\bm{\rho}}(H_{\bm{p}\bm{p}}^{-1})^{\vee}\right| \lesssim~ (1+|\xi|)^{-d-|\bm{\rho}|},
\]
we simply note that $D_\gamma (H_{\bm{p}\bm{p}}^{-1})^{\vee} \in
L^1(\mathcal{B})$ for any $\gamma$. (In fact the decay is at least
super-algebraic, but this will be dominated by other
$(1+|\xi|)^{-d-|\bm{\rho}|}$ terms later in the proof.)

The first part of the estimate,
\[
   \left|D_{\bm{\rho}}(H_{\bm{p}\bm{p}}^{-1}H_{\bm{p}0}Q^{-1}H_{0\bm{p}}H_{\bm{p}\bm{p}}^{-1})^{\vee}\right| \lesssim~ (1+|\xi|)^{-d-|\bm{\rho}|},
\]
can be obtained using a procedure very similar to that in the proof of~\eqref{hessianDecay3}, that
is, by comparing $Q^{-1}$ with $M^{-1}$ and $H_{0\bm{p}}$ with $J_{0\bm{p}}$.
Briefly, while $\partial_\gamma Q^{-1} = \mathcal{O}(k^{-2-|\gamma|})$, the blocks $H_{\bm{p}0}$ and
$H_{0\bm{p}}$ contribute two additional powers of $k$ which in real-space terms
translates to the improvement of the decay estimate \eqref{hessianDecay3}
over \eqref{hessianDecay1}.
\end{proof}

\subsection{Decay of Displacement and Shifts}
Using the decay estimates on the Hessian from the previous section and the
residual decay estimates on the linearized equation~\eqref{lin_equation}, we now
establish the desired decay rates for the displacement field $U^\infty$ and
shift fields $\bm{p}^\infty$.  Recall that from the linearized
equation~\eqref{lin_equation}, we have
\begin{align*}
\hat{U}^\infty =~& Q^{-1}F - Q^{-1}H_{0\bm{p}}H_{\bm{p}\bm{p}}^{-1}\bm{g} \\
\hat{\bm{p}}^\infty =~& -H_{\bm{p}\bm{p}}^{-1}H_{\bm{p}0}Q^{-1} F + H_{\bm{p}\bm{p}}^{-1}H_{\bm{p}0}Q^{-1}H_{0\bm{p}}H_{\bm{p}\bm{p}}^{-1}\bm{g} + H_{\bm{p}\bm{p}}^{-1}\bm{g}.
\end{align*}
Observe also that $F = \mathcal{O}(k)$ and $g_\eta = \mathcal{O}(1)$ from~\eqref{residual}. By taking inverse Fourier transforms, we obtain
\begin{align*}
U^\infty =~& (Q^{-1})^{\vee}*\check{F} - (Q^{-1}H_{0\bm{p}}H_{\bm{p}\bm{p}}^{-1})^{\vee}*\check{\bm{g}} \\
\bm{p}^\infty =~&  (-H_{\bm{p}\bm{p}}^{-1}H_{\bm{p}0}Q^{-1})^{\vee}* \check{F} + (H_{\bm{p}\bm{p}}^{-1}H_{\bm{p}0}Q^{-1}H_{0\bm{p}}H_{\bm{p}\bm{p}}^{-1})^{\vee}*\check{\bm{g}} + (H_{\bm{p}\bm{p}}^{-1})^{\vee}*\check{\bm{g}}
\end{align*}
For notational convenience, we set $A := Q^{-1}$ and $B = Q^{-1}H_{0\bm{p}}H_{\bm{p}\bm{p}}^{-1}$ and rewrite the first of these as
\begin{align*}
U^\infty(\ell)
   &=~ A*\check{F} - B*\check{\bm{g}}
   = \sum_{\xi \in \mathcal{L}} \bigg( A(\ell -\xi)\check{F}(\xi) + \sum_{\alpha = 0}^{S-1} B_\alpha(\ell-\xi) \check{g}_\alpha(\xi) \bigg)  \\
&=~ \sum_{\xi \in \mathcal{L}}\sum_{\triple \in \mathcal{R}}A(\ell -\xi)D_{\rho}f_{\triple}(\xi) \\
&\qquad \qquad + \sum_{\xi \in \mathcal{L}} \sum_{\alpha = 0}^{S-1}  B_\alpha(\ell-\xi)\left[\sum_{\rho\in \mathcal{R}_0}\sum_{\beta = 0}^{S-1} \left(f_{(\rho\beta\alpha)}(\xi+\rho) - f_{\triple}(\xi)\right)\right] \\
&=~ \sum_{\xi \in \mathcal{L}}\sum_{\triple \in \mathcal{R}}A(\ell -\xi) D_{\rho}f_{\triple}(\xi) \\
&\qquad \qquad + \sum_{\xi \in \mathcal{L}}
      \bigg[\sum_{\triple \in \mathcal{R}}  B_\alpha(\ell-\xi)\big(f_{(\rho\beta\alpha)}(\xi+\rho) - f_{\triple}(\xi)\big)\bigg].
\end{align*}
In a similar manner, we may rewrite the second of these as
\begin{align*}
p_\alpha^\infty(\ell) =~& \sum_{\beta}\sum_{\xi \in \mathcal{L}}(-B^*)_\beta^{\vee}(\ell -\xi)\sum_{\tripleTau}D_\tau f_{\tripleTau}(\xi)  \\
&+~ \sum_\beta\sum_{\xi \in \mathcal{L}} (H_{\bm{p}\bm{p}}^{-1}H_{\bm{p}0}Q^{-1}H_{0\bm{p}}H_{\bm{p}\bm{p}}^{-1})_\beta^{\vee}(\ell-\xi)\check{g}_\beta(\xi) + \sum_{\beta}\sum_{\xi\in\mathcal{L}}(H_{\bm{p}\bm{p}}^{-1})_\beta^{\vee}*\check{g}_\beta(\xi).
\end{align*}

We are now ready to prove our main result, Theorem~\ref{main_thm}.
%\begin{theorem}[Decay of Displacements and Shifts]
%There exists $L > 0$ such that
%\begin{align*}
%|D_{\bm{\rho}} U^\infty(\xi)| \lesssim~& (1 + |\xi|)^{1-d-|\bm{\rho}|}, \quad \forall |\xi| \geq L, |\bm{\rho}| \leq 3, \\
%|D_{\bm{\rho}} p_\alpha^\infty(\xi)| \lesssim~& (1 + |\xi|)^{-d-|\bm{\rho}|}, \quad \forall |\xi| \geq L, |\bm{\rho}| \leq 2.
%\end{align*}
%\end{theorem}
\begin{proof}[Proof of Theorem~\ref{main_thm}]

   {\it Part I: proof of lowest-order decay. }
We begin by proving the conclusion of the theorem for $D_\tau U^\infty$ (that
is, $\bm{\rho} = \rho$ with $|\bm{\rho}| = 1$) and $p_\alpha^\infty$ ($|\bm{\rho}| = 0$) and will follow
the same method as~\cite[Section 6]{Ehrlacher2013}. The main idea is to prove
the result similar to how one would prove that the convolution of two functions
with known decay will decay at the slower of the two rates: we split the
convolution over an inner set and an outer set and then use the relevant decay
properties on each set.  Here, the decay of the Green's functions is governed by
Theorem~\ref{hess_thm}, and the decay of the residual is governed by
Corollary~\ref{decay_cor1}.

To this end define the translation operator $T_\rho q_\beta(\xi) := q_\beta(\xi + \rho)$, and set
\[
   w(r) = \sup_{|\ell| \geq r} |DU^\infty(\ell)|, \quad
   q_\beta(r) = \sup_{|\ell| \geq r}\max_{\rho \in \mathcal{R}_1}|T_\rho p_\beta^\infty(\ell)|,
   \quad  q(r) = \sup_\beta q_\beta(r),
\]
%Note
%\begin{align*}
%U^\infty(\ell) =~& \sum_{\xi \in \mathcal{L}}\sum_{\rho\alpha\beta}A(\ell -\xi)D_{\rho}f_{\rho \alpha \beta}(\xi) \\
%&\quad+ \sum_\xi\left[\sum_\alpha\sum_{\rho}\sum_{\beta} B_\alpha(\ell-\xi)\left(f_{\rho\beta\alpha}(\xi+\rho) - f_{\rho\alpha\beta}(\xi)\right)\right] \\
%p_\alpha^\infty(\ell) =~& \sum_{\beta}\sum_{\xi \in \mathcal{L}}(-H_{\bm{p}\bm{p}}^{-1}H_{\bm{p}0}H_{00}^{-1})_\beta^{\vee}(\ell -\xi),\check{F}(\xi)  \\
%&+~ \sum_\beta\sum_{\xi \in \mathcal{L}} (H_{\bm{p}\bm{p}}^{-1}H_{\bm{p}0}H_{00}^{-1}H_{0\bm{p}}H_{\bm{p}\bm{p}}^{-1})_\beta^{\vee}(\ell-\xi)\check{g}_\beta(\xi) + \sum_{\beta}\sum_{\xi\in\mathcal{L}}(H_{\bm{p}\bm{p}}^{-1})_\beta^{\vee}*\check{g}_\beta(\xi).
%\end{align*}
%so
and note that
\begin{equation}\label{w1}
\begin{split}
w(2r) &= \sup_{|\ell| \geq 2r} \max_{\tau \in \mathcal{R}_1}
            |D_\tau U^\infty(\ell)| \\
      &\hspace{-1cm} = \sup_{|\ell| \geq 2r} \max_{\tau \in \mathcal{R}_1}
      \bigg|\sum_{\xi \in \mathcal{L}}\sum_{\triple \in \mathcal{R}}\bigg(
      D_\tau D_\rho A(\ell -\xi)f_{\triple}(\xi)  \\[-4mm]
      & \hspace{4cm} +
         D_\tau B_\alpha(\ell-\xi) \left(f_{(\rho\beta\alpha)}(\xi+\rho) - f_{\triple}(\xi)\right)  \bigg) \bigg| \\
         &\hspace{-1cm} = \sup_{|\ell| \geq 2r} \max_{\tau \in \mathcal{R}_0}
         \bigg| \sum_{\xi \in \mathcal{L}}\sum_{\triple \in \mathcal{R}} \bigg(
         D_\tau D_\rho A(\xi)f_{\triple}(\ell-\xi)  \\[-4mm]
      &\hspace{4cm} +
            D_\tau B_\alpha(\xi+\rho)f_{(\rho\beta\alpha)}(\ell-\xi)
            - D_\tau B_\alpha(\xi) f_{\triple}(\ell-\xi)  \bigg) \bigg|.
            %
% &= \sup_{|\ell| \geq 2r} \max_{\tau \in \mathcal{R}_0}
% \left|\sum_{\xi \in \mathcal{L}}\sum_{\triple}
% D_\tau D_\rho A(\ell -\xi)f_{\triple}(\xi)\right. \\
% &\qquad \qquad \qquad + \left.\sum_\xi\sum_\alpha\sum_{\rho}\sum_{\beta} D_\tau B_\alpha(\ell-\xi)\left(f_{(\rho\beta\alpha)}(\xi+\rho) - f_{\triple}(\xi)\right) \right| \\
% %
% &= \sup_{|\ell| \geq 2r} \max_{\tau \in \mathcal{R}_0}
% \left|\sum_{\xi \in \mathcal{L}}\sum_{\triple}D_\tau D_\rho A(\xi)f_{\triple}(\ell-\xi)\right. \\
% &\qquad \qquad \qquad
% + \left.\sum_\xi\sum_\alpha\sum_{\rho}\sum_{\beta} D_\tau B_\alpha(\xi)\left(f_{(\rho\beta\alpha)}(\ell-\xi+\rho) - f_{\triple}(\ell-\xi)\right) \right|.
\end{split}
\end{equation}
By Theorem~\ref{hess_thm}, $|D_\tau D_\rho A(\xi)| \lesssim (1+|\xi|)^{-d}$ and $|D_\tau B_\alpha(\xi)| \lesssim (1+|\xi|)^{-d}$, and by Corollary~\ref{decay_cor1}, $|f(\xi)| \lesssim |D\bm{u}^\infty(\xi)|^2$.  Employing these estimates in~\eqref{w1}, we then get
\begin{equation*} % \label{w2}
\begin{split}
   w(2r) &\lesssim \sup_{|\ell| \geq 2r}(1+|r|)^{-d}\sum_{|\xi| \geq r} |D\bm{u}^\infty(\ell-\xi)|^2 + \sup_{|\ell| \geq 2r}\sum_{|\xi| \leq r} (1+|\xi|)^{-d} |D\bm{u}^\infty(\ell-\xi)|^2 \\
&\lesssim (1+|r|)^{-d} \|\bm{u}^\infty\|_\a^2
+ \big(w(r)^{3/2} + q(r)^{3/2}\big)
 \sum_{\xi \in \mathcal{L}} (1+|\xi|)^{-d}|D\bm{u}^\infty(\ell-\xi)|^{1/2}.
 % \Big[\sup_{\rho}|D_\rho U^\infty(\ell-\xi)|^2 +  \sup_{\rho,\alpha}|T_\rho p_\alpha(\ell-\xi)|^2\Big].
\end{split}
\end{equation*}
Since $|D\bm{u}^\infty| \in \ell^2$ it follows that
$(1+|\xi|)^{-d}|D\bm{u}^\infty(\ell-\xi)|^{1/2}$ is summable, hence we obtain
%
% After splitting the summation up in this way, we have $|\ell-\xi| \geq r$ for $|\xi|\leq r$ so
% \begin{equation*}
% \begin{split}
% w(2r) \lesssim~& (1+|r|)^{-d}\|\bm{u}^\infty\|_\a^2 \\
% &~+ w(r)\sqrt{w(r)}\sup_{|\ell| \geq 2r}\sum_{|\xi| \leq r} \big[(1+|\xi|)^{-d} + \sup_{\rho}(1+|\xi+\rho|)^{-d}\big]\sup_{\rho}|D_\rho U^\infty(\ell-\xi)|^{1/2} \\
% &~+ q(r)\sqrt{q(r)} \sup_{|\ell| \geq 2r}\sum_{|\xi| \leq r}  \big[(1+|\xi|)^{-d} + \sup_{\rho}(1+|\xi+\rho|)^{-d}\big]\sup_{\rho,\alpha}|T_\rho p_\alpha(\ell-\xi)|^{1/2},
% \end{split}
% \end{equation*}
% and then applying a H\"older inequality with conjugate exponents $\frac{4}{3}$ and $4$ and using the fact that $\sum_{\xi} (1+|\xi|)^{-d}$ is $\ell^{4/3}$ summable yields
%
\begin{equation} \label{w1-fin}
\begin{split}
w(2r)  \lesssim~& (1+|r|)^{-d} + w(r)\sqrt{w(r)} + q(r)\sqrt{q(r)}.
%\leq~& (1+|\ell|)^{-d}\|\bm{u}^\infty\|^2 + \sup_{|\xi| \leq |\ell|, \rho}|D_\rho U^\infty(\xi)|\sum_{|\xi| \leq |\ell|}(1+|\xi|)^{-d}
\end{split}
\end{equation}

By analogous computations and employing the remaining decay rates of Theorem~\ref{hess_thm},
\begin{equation} \label{w2}
\begin{split}
q_\beta(2r) &= \sup_{|\ell| \geq 2r} \max_{\rho \in \mathcal{R}_1} |T_\rho p_\beta^\infty(\ell)| \\
&= \sup_{|\ell| \geq 2r}\sum_{\beta}\sum_{\xi \in \mathcal{L}}T_\rho(-H_{\bm{p}\bm{p}}^{-1}H_{\bm{p}0}H_{00}^{-1})_\beta^{\vee}(\xi)\check{F}(\ell-\xi)  \\
&\qquad + \sup_{|\ell| \geq 2r}\sum_\beta\sum_{\xi \in \mathcal{L}} T_\rho(H_{\bm{p}\bm{p}}^{-1}H_{\bm{p}0}H_{00}^{-1}H_{0\bm{p}}H_{\bm{p}\bm{p}}^{-1})_\beta^{\vee}(\xi)\check{g}_\beta(\ell-\xi) \\
&\qquad + \sup_{|\ell| \geq 2r}\sum_{\beta}\sum_{\xi\in\mathcal{L}}T_\rho(H_{\bm{p}\bm{p}}^{-1})_\beta^{\vee}(\xi)\check{g}_\beta(\ell-\xi) \\
&\lesssim (1+|r|)^{-d} + w(r)\sqrt{w(r)} + q(r)\sqrt{q(r)}.
\end{split}
\end{equation}
Combining equations~\eqref{w1-fin} and~\eqref{w2}, we have
\begin{align*}
w(2r) \leq~& C_1(1+|r|)^{-d} + C_1w(r)\sqrt{w(r)} + C_1q(r)\sqrt{q(r)}, \\
q(2r) \leq~& C_2(1+|r|)^{-d} + C_2w(r)\sqrt{w(r)} + C_2q(r)\sqrt{q(r)}.
\end{align*}

Applying Step 2 of~\cite[Lemma 6.3]{Ehrlacher2013} to $v(r) = r^d(w(r) + q(r))$
we deduce that there exists a constant $C$ such that
\[
   |r^d(w(r) + q(r))| \leq~ C \quad \forall \, r > 0.
\]
% and hence,
% \[
%    \sup_{|\ell| \geq r} |DU^\infty(\ell)| + \sup_{|\ell| \geq r, \beta}|q_\beta^\infty(\ell)| \lesssim~ r^{-d}   \quad \forall \, r > 0,
% \]
which completes the proof for the lowest-order decay,
\begin{align*}
|D_{\rho} U^\infty(\xi)| \lesssim~& (1 + |\xi|)^{-d}
\qquad \text{and} \qquad
|p_\alpha^\infty(\xi)| \lesssim (1 + |\xi|)^{-d}.
\end{align*}

%Next, we use induction on $|\rho|$ having already proven the base case of $|\rho| = 0$. Inductively assume then that
%\begin{align*}
%|D_\tau D_{\bm{\rho}} U^\infty(\xi)| \lesssim~& (1 + |\xi|)^{1-d-|\rho|-1} \\
%|D_{\bm{\rho}} p_\alpha^\infty(\xi)| \lesssim~& (1 + |\xi|)^{-d-|\rho|}
%\end{align*}
%With this hypothesis, we easily have
%\begin{align*}
%|D_\sigma D_\tau D_{\bm{\rho}} U^\infty(\xi)| \lesssim~& (1 + |\xi|)^{1-d-|\rho|-1} \\
%|D_\omega D_{\bm{\rho}} p_\alpha^\infty(\xi)| \lesssim~& (1 + |\xi|)^{-d-|\rho|}.
%\end{align*}

{\it Part II: proof of higher-order decay. } Let $\bm{\rho} \in \mathcal{R}_0^t,
      t \in \{2, 3 \}$, then we have
\begin{align*}
   D_{\bm{\rho}}U^\infty(\ell) =~& \sum_{\xi \in \mathcal{L}} \bigg( D_{\bm{\rho}}A(\xi)\check{F}(\ell-\xi) + \sum_{\alpha = 0}^{S-1}  D_{\bm{\rho}}B_\alpha(\xi)\check{g}_\alpha(\ell-\xi) \bigg).
\end{align*}
The decay rates established in {\it Part I} of the proof in particular
entail that
\begin{equation*}
   |DU^\infty(\xi)| \lesssim (1+|\xi|)^{-d}
\end{equation*}
hence Theorems~\ref{decay_cor1} implies that
\begin{equation}
   \label{eq:tunafish1}
   |\check{F}(\xi)| + |\check{\bf{g}}(\xi)|
   \lesssim |f(\xi)| \lesssim (1+|\xi|)^{-2d}.
\end{equation}

Using also the decay estimates for the Green's matrix, from
Theorem~\ref{hess_thm}, we continue to estimate
\begin{align}
      \notag
\big|D_{\bm{\rho}}U^\infty(\ell)\big| %=~& \sum_{|\xi| \leq 1/2|\ell|}D_{\bm{\rho}}A(\xi)\check{F}(\ell-\xi) + D_{\bm{\rho}}B_\alpha(\xi)\check{g}_\alpha(\ell-\xi) \\
%&\qquad+~ \sum_{|\xi| \geq 1/2|\ell|}D_{\bm{\rho}}A(\xi)\check{F}(\ell-\xi) + D_{\bm{\rho}}B_\alpha(\xi)\check{g}_\alpha(\ell-\xi) \\
\lesssim~& \sum_{|\xi| \leq 1/2|\ell|}(1+|\xi|)^{1-d-|\rho|}(1+|\ell-\xi|)^{-2d} \\
\notag
&\qquad+~ \sum_{|\xi| \geq 1/2|\ell|}(1+|\xi|)^{1-d-|\rho|}(1+|\ell-\xi|)^{-2d} \\
\notag
\lesssim~& (1+|\ell|)^{-2d}\sum_{|\xi| \leq 1/2|\ell|}(1+|\xi|)^{1-d-|\rho|} \\
\notag
&\qquad+~ (1+|\ell|)^{1-d-|\rho|}\sum_{|\xi| \geq 1/2|\ell|}(1+|\ell-\xi|)^{-2d} \\
   \label{eq:tunafish2}
\lesssim~&  (1+|\ell|)^{-2d} + (1+|\ell|)^{1-d-|\rho|},
\end{align}
which completes the proof of the first estimate in \eqref{decay1}.

To establish the corresponding higher-order decay for the shifts,
let $\bm{\rho} \in \mathcal{R}_0^{t}, t \in \{1,2\}$, then
\begin{align*}
&D_{\bm{\rho}}p_\alpha^\infty(\ell) =~ \sum_{\beta}\sum_{\xi \in \mathcal{L}}D_{\bm{\rho}}(-H_{\bm{p}\bm{p}}^{-1}H_{\bm{p}0}Q^{-1})_\beta^{\vee}(\xi)\check{F}(\ell-\xi)  \\
&+~ \sum_\beta\sum_{\xi \in \mathcal{L}} D_{\bm{\rho}}(H_{\bm{p}\bm{p}}^{-1}H_{\bm{p}0}Q^{-1}H_{0\bm{p}}H_{\bm{p}\bm{p}}^{-1})_\beta^{\vee}(\xi)\check{g}_\beta(\ell-\xi) + \sum_{\beta}\sum_{\xi\in\mathcal{L}}D_{\bm{\rho}}(H_{\bm{p}\bm{p}}^{-1})_\beta^{\vee}*\check{g}_\beta(\ell-\xi).
\end{align*}
As in the estimate for $D_{\bfrho} U^\infty$, we insert the Green's matrix decay
estimate from Theorem~\ref{hess_thm} and \eqref{eq:tunafish1}, and then argue
precisely as in \eqref{eq:tunafish2} to obtain the second estimate
in \eqref{decay1}.
%
% We perform the same decomposition:
% \begin{align*}
% &D_{\bm{\rho}}p_\alpha^\infty(\ell) =~ \sum_{\beta}\sum_{|\xi| \leq 1/2|\ell|}D_{\bm{\rho}}(-H_{\bm{p}\bm{p}}^{-1}H_{\bm{p}0}Q^{-1})_\beta^{\vee}(\xi)\check{F}(\ell-\xi)  \\
% &+ \sum_\beta\sum_{|\xi| \leq 1/2|\ell|} D_{\bm{\rho}}(H_{\bm{p}\bm{p}}^{-1}H_{\bm{p}0}H_{00}^{-1}H_{0\bm{p}}H_{\bm{p}\bm{p}}^{-1})_\beta^{\vee}(\xi)\check{g}_\beta(\ell-\xi) + \sum_{\beta}\sum_{|\xi| \leq 1/2|\ell|}D_{\bm{\rho}}(H_{\bm{p}\bm{p}}^{-1})_\beta^{\vee}*\check{g}_\beta(\ell-\xi) \\
% &\sum_{\beta}\sum_{|\xi| \geq 1/2|\ell|}D_{\bm{\rho}}(-H_{\bm{p}\bm{p}}^{-1}H_{\bm{p}0}Q^{-1})_\beta^{\vee}(\xi)\check{F}(\ell-\xi)  \\
% &+ \sum_\beta\sum_{|\xi| \geq 1/2|\ell|} D_{\bm{\rho}}(H_{\bm{p}\bm{p}}^{-1}H_{\bm{p}0}Q^{-1}H_{0\bm{p}}H_{\bm{p}\bm{p}}^{-1})_\beta^{\vee}(\xi)\check{g}_\beta(\ell-\xi) + \sum_{\beta}\sum_{|\xi| \geq 1/2|\ell|}D_{\bm{\rho}}(H_{\bm{p}\bm{p}}^{-1})_\beta^{\vee}*\check{g}_\beta(\ell-\xi),
% \end{align*}
% and then we use our standard bounds to get
% \begin{align*}
% &\big|D_{\bm{\rho}}p_\alpha^\infty(\ell)\big| \\
% &\lesssim~ \sum_{\beta}\sum_{|\xi| \leq 1/2|\ell|}(1+|\xi|)^{-d-|\rho|}(1+|\ell-\xi|)^{-2d} + \sum_{\beta}\sum_{|\xi| \geq 1/2|\ell|}(1+|\xi|)^{-d-|\rho|}(1+|\ell-\xi|)^{-2d} \\
% &\lesssim~ (1 + 1/2|\ell|)^{-2d}\sum_{\beta}\sum_{|\xi| \leq 1/2|\ell|}(1+|\xi|)^{-d-|\rho|} + (1+|\ell|)^{-d-|\rho|}\sum_{\beta}\sum_{|\xi| \geq 1/2|\ell|}(1+|\ell-\xi|)^{-2d} \\
% &\lesssim~ \sum_\beta (1 + 1/2|\ell|)^{-2d} + \sum_\beta (1+1/2|\ell|)^{-d-|\rho|},
% \end{align*}
% and so we get the result for $|\rho| = 1,2$.
\end{proof}

\section{Discussion}\label{discussion}

We have extended the model formulation and analysis (decay of
discrete elastic fields) for point defects embedded in a homogeneous crystalline
solid from the Bravais lattice case \cite{Ehrlacher2013} to multilattices.
While, at a conceptual level, the arguments remained fairly similar, numerous
modifications were required in accounting for the shift degrees of freedom, in
particular an extension of the decay estimates for the lattice Green's matrix to
the multilattice case. Our results build a foundation for the numerical analysis
of coarse-graining schemes for multilattices, in particular an analysis of
atomistic/continuum blending schemes \cite{olsonUnpub}.

To conclude we briefly mention some important extensions: (1) To include
dislocations we need to replace the reference lattice as the predictor
configuration with a linearised elasticity solution. We anticipate that
following the ideas from \cite{Ehrlacher2013} but replacing the simple lattice
Cauchy--Born model for the computation of the predictor displacement with the
classical multilattice Cauchy--Born model \eqref{eq:classical-cb} should be
sufficient to carry out this extension.

(2) A second problem of interest is the extension of our analysis to ionic
crystals. Here, long-range interactions play a crucial role, and it is
at this point largely unclear to what extent our results generalise.

%  linearize the problem.  For point defects, we were able
% to take the original multilattice, but for dislocations, this choice is not
% sufficient.  Instead, a modified solution of the Cauchy--Born continuum
% elasticity problem is used~\cite{Ehrlacher2013}.
(3) Finally, a problem of current interest is the application of our results to
defects in bilayer materials~\cite{alden2013strain}, where two or more
multilattice crystals are stacked on top of each other.  By considering the top
layer to be shifted relative to the bottom layer, our current results extend to that
case as long as the multilattices in each layer are the same (or, more generally,
have a common periodic cell).  However, this
does not allow for important effects such as disregistry to be modeled where the
lattice constants in each layer differ by an irrational factor~\cite{cazeaux2016}.  These effects
would require a different analysis due to lack of periodicity and lack of
continuum model to compare the atomistic Green's function too.

\appendix

\section{Proofs and Additional Results}

\subsection{Density of Test Functions}
\label{sec:dense}

Here we prove density of the test function space.

\begin{lemma}\label{lem:dense}
The quotient space $\bm{\mathcal{U}}_0$ is dense in $\bm{\mathcal{U}}  = \mathcal{U}/ \mathbb{R}^n$.
\end{lemma}

\begin{proof}
The proof is a slight modification of~\cite[Theorem 2.1]{suli2012} taking into account both the interpolation operator and additional shift vectors.  We only provide a brief sketch of the proof; for a related proof in the context of a simple lattice, see~\cite[Lemma 1.8]{olson2015}.

Let $\eta$ be a smooth bump function with support in $B_{1}(0)$ and equal to one on $B_{3/4}(0)$, and for $R > 0$, let $\eta_R(x) := \eta(x/R)$ and $A_R := {\rm supp}(\nabla (I \eta_R))$. Next, for $\bm{u} \in \bm{\mathcal{U}}$, define the truncation operator $T_{R}\bm{u} = (T_Ru_\alpha)_{\alpha = 0}^{S-1}$ by
\begin{equation*}\label{trunc}
T_{R}u_\alpha(x) = \eta_R(x) \big(I u_\alpha - \frac{1}{|A_R|}\int\limits_{A_R} Iu_0\, dx\big),
\end{equation*}
where $|A_R|$ represents the measure of $A_R$.  Then define
\begin{align*}
&\Pi_R \bm{u} := (\Pi_R u_\alpha)_{\alpha = 0}^{S-1}, \\
&\Pi_R u_\alpha := I (T_Ru_\alpha).
\end{align*}
Clearly $\Pi_R \bm{u} \in \bm{\mathcal{U}}_0$, and so we need to show $\Pi_R\bm{u} - \bm{u} \to 0$ as $R \to \infty$.  Using the definition of $\Pi_R$, it is straightforward to show
\begin{equation}\label{fox1}
\begin{split}
&\|\nabla \Pi_R u_\alpha - \nabla I u_\alpha \|_{L^2(\mathbb{R}^d)} =~ \|\nabla I T_R u_\alpha - \nabla I u_\alpha \|_{L^2(\mathbb{R}^d)} \\
&\lesssim~ \|\nabla (I(\eta_R(I u_\alpha - \frac{1}{|A_R|}\int\limits_{A_R} Iu_0\, dx))) - \nabla(I\eta_R (Iu_\alpha -\frac{1}{|A_R|}\int\limits_{A_R} Iu_0\, dx))  \\
&\qquad  + \nabla(I\eta_R(Iu_\alpha -\frac{1}{|A_R|}\int\limits_{A_R} Iu_0\, dx)) -\nabla I u_\alpha \|_{L^2(\mathbb{R}^d)} \\
&\lesssim~ \|\nabla (I (\eta_R u_\alpha)) - \nabla (I\eta_R Iu_\alpha) \|_{L^2(A_R)} \\
&~+ \| (Iu_\alpha - \frac{1}{|A_R|}\int\limits_{A_R} Iu_0\, dx)\nabla I\eta_R^\transpose\|_{L^2(A_R)} + \|(I\eta_R - 1)\nabla Iu_\alpha \|_{L^2(A_R)}  \\
&\qquad\qquad\quad+ \|\nabla I u_\alpha \|_{L^2(\mathbb{R}^d\setminus B_R)}.
\end{split}
\end{equation}
Clearly, the latter two terms tend to zero as $R \to \infty$ since $\nabla u_\alpha \in L^2(\mathbb{R}^d)$.

By splitting the first term into a sum over triangles and using standard
interpolation estimates on each triangle, the first term in \eqref{fox1} can
also be seen to go to zero as $R \to \infty$:
\begin{align*}
&\hspace{-1.5cm} \|\nabla (I(\eta_R u_\alpha)) - \nabla (I\eta_R Iu_\alpha) \|_{L^2(A_R)}^2
=~ \|\nabla I(I\eta_R I u_\alpha) - \nabla (I\eta_R Iu_\alpha) \|_{L^2(A_R)}^2 \\
&=~ \sum_{T \in \mathcal{T}_\a, T \cap A_R \neq \emptyset}
   \|\nabla I(I\eta_R I u_\alpha) - \nabla (I\eta_R Iu_\alpha) \|_{L^2(T)}^2 \\
&\lesssim~ \sum_{T \in \mathcal{T}_\a, T \cap A_R \neq \emptyset}
   \|\nabla^2 (I\eta_R I u_\alpha)\|_{L^2(T)}^2
   =~ 2\sum_{T \in \mathcal{T}_\a, T \cap A_R \neq \emptyset}
   \| \nabla I\eta_R \nabla I u_\alpha^\transpose  \|_{L^2(T)}^2  \hspace{-1cm}  \\
&\lesssim~ \frac{1}{R^2} \sum_{T \in \mathcal{T}_\a, T \cap A_R \neq \emptyset}
   \|\nabla I u_\alpha \|_{L^2(T)}^2
   \lesssim~ \frac{1}{R^2} \|\nabla Iu_\alpha \|_{L^2(A_R)}^2
 \to 0 \quad \text{as $R \to \infty$,} \hspace{-1cm}
% &\lesssim~ \sum_{T \in \mathcal{T}_\a, T \cap A_R \neq \emptyset}
%    \|\nabla (I u_\alpha) (\nabla (\eta(x/R)))^\transpose\|_{L^2(T)}^2 + \|\nabla (I u_\alpha) (\nabla (\eta(x/R)))^\transpose - \nabla (I u_\alpha) (\nabla I(\eta(x/R)))^\transpose  \|_{L^2(T)}^2 \\
% &\lesssim~ \frac{1}{R^2} \|\nabla (Iu_\alpha) (\nabla \eta)(x/R)\|_{L^2(A_R)}^2 + \frac{1}{R^4}\|\nabla Iu_\alpha\|_{L^2(A_R)}^2 \|(\nabla^2 \eta)(x/R)\|_{L^2(A_R)}^2 \to 0.
\end{align*}
where we used $\| \nabla I\eta_R \|_{L^\infty} \lesssim \| \nabla \eta_R \|_{L^\infty}
\lesssim  R^{-1}$ in the second inequality.

The second term in~\eqref{fox1} can also be seen to converge to zero after using
the Poincar\'e inequality and the fact that the Poincar\'e constant for $A_R$ is
bounded by a constant multiple of $R$.  Specifically,
\begin{align*}
 \| (Iu_\alpha - \frac{1}{|A_R|}\int\limits_{A_R} Iu_0\, dx)\nabla (I\eta(x/R))^T\|_{L^2(A_R)} \lesssim~ \|\nabla Iu_\alpha\|_{L^2(A_R)} + \frac{1}{R} \| Iu_\alpha - Iu_0 \|_{L^2(A_R)},
\end{align*}
which clearly tends to zero.
\end{proof}

\subsection{Proof of Theorem~\ref{well_defined}}

\label{sec:proof-well_defined}

% \begin{proof}
As the summations defining $\mathcal{E}^\a_{\rm hom}(\bm{u})$ and
$\mathcal{E}^\a(\bm{u})$ differ only on the finite set where $V_\xi \not\equiv
V$, we need only show that $\mathcal{E}^\a_{\rm hom}(\bm{u})$ is well-defined.
We prove this along the lines of~\cite{theil2012}[Theorem 2.8]; we will
construct an auxiliary energy functional $\bar{\mathcal{E}}^\a_{\rm hom}$ which
is ${\rm C}^3$ and show that $\mathcal{E}^\a_{\rm hom}$ and
$\bar{\mathcal{E}}^\a_{\rm hom}$ are equal on the dense subset
$\bm{\mathcal{U}}_0$.

To that end, define
\[
\bar{\mathcal{E}}^\a_{\rm hom}(\bm{u}) := \sum_{\xi \in \mathcal{L}} \big[V(D\bm{u}(\xi)) - \sum_{\triple \in \mathcal{R}}V_{,\triple}(D\bm{y}(\xi))\cdot D_{\triple}\bm{u}(\xi)\big].
\]
Using a Taylor expansion of the site potential about $D\bm{y}(\xi)$ and a bound on the second derivatives of $V$,
\begin{align*}
|\bar{\mathcal{E}}^\a_{\rm hom}(\bm{u})| \lesssim~& \sum_{\xi \in \mathcal{L}} \sum_{\triple \in \mathcal{R}} \sum_{\tripleTau \in \mathcal{R}} |D_{\triple}\bm{u}(\xi)|\cdot |D_{\tripleTau}\bm{u}(\xi)| \\
\lesssim~& \Big\{\sum_{\xi \in \mathcal{R}}\sum_{\triple \in \mathcal{R}}|D_{\triple}\bm{u}(\xi)|^2  \Big\}^{1/2} \Big\{\sum_{\xi \in \mathcal{R}}\sum_{\tripleTau \in \mathcal{R}}|D_{\tripleTau}\bm{u}(\xi)|^2 \Big\}^{1/2}
\leq \|\bm{u}\|_{\a_1}^2.
\end{align*}
Since $\bar{\mathcal{E}}^\a_{\rm hom}$ is clearly invariant with respect to addition by constants, this shows $\bar{\mathcal{E}}^\a_{\rm hom}$ is well-defined on the quotient space $\bm{\mathcal{U}}$.

To show $\bar{\mathcal{E}}^\a_{\rm hom}(\bm{u})$ is differentiable, we again use a Taylor expansion and bound on the second derivative of $V$ to observe
\begin{align*}
&\bar{\mathcal{E}}^\a_{\rm hom}(\bm{u} + \bm{v}) - \bar{\mathcal{E}}^\a_{\rm hom}(\bm{u}) - \sum_{\xi \in \mathcal{L}}\sum_{\triple \in \mathcal{R}} \big[V_{,\triple}(D\bm{u}(\xi))\cdot D_{\triple}\bm{v}(\xi)
\\ &\qquad\qquad\qquad\qquad\qquad\qquad
+ V_{,\triple}(D\bm{y}(\xi))\cdot D_{\triple}\bm{v}(\xi)\big] \\
&\lesssim~ \sum_{\xi \in \mathcal{L}} \sum_{\triple \in \mathcal{R}} \sum_{\tripleTau \in \mathcal{R}} |D_{\triple}\bm{v}(\xi)|\cdot |D_{\tripleTau}\bm{v}(\xi)| \lesssim~ \|\bm{v}\|_{\a_1}^2.
\end{align*}
The first Fr\'echet derivative of $\bar{\mathcal{E}}^\a_{\rm hom}$ is thus defined by
\[
\< \delta \bar{\mathcal{E}}^\a_{\rm hom}(\bm{u}), \bm{v}\> = \sum_{\xi \in \mathcal{L}}\sum_{\triple \in \mathcal{R}} \big[V_{,\triple}(D\bm{u}(\xi))\cdot D_{\triple}(\bm{v}(\xi)) + V_{,\triple}(D\bm{y}(\xi))\cdot D_{\triple}\bm{v}(\xi)\big].
\]

To prove that $\delta \bar{\mathcal{E}}^\a_{\rm hom}(\bm{u})$ is differentiable, we again employ a Taylor expansion and a bound on the third derivative of $V$
\begin{align*}
&\<\delta \bar{\mathcal{E}}^\a_{\rm hom}(\bm{u} + \bm{w}) - \delta \bar{\mathcal{E}}^\a_{\rm hom}(\bm{u}), \bm{v}\> \\
&\qquad - \sum_{\xi \in \mathcal{L}}\sum_{\triple \in \mathcal{R}}\sum_{\tripleTau \in \mathcal{R}} \big[D_{\tripleTau}(\bm{w}(\xi))\big]^{\transpose}V_{,\triple\tripleTau}(D\bm{u}(\xi))\big[D_{\triple}\bm{v}(\xi)\big] \\
&\lesssim~ \sum_{\xi \in \mathcal{L}}\sum_{\triple \in \mathcal{R}}\sum_{\tripleTau \in \mathcal{R}}\sum_{\tripleSig \in \mathcal{R}} |D_{\triple}\bm{v}(\xi)|\cdot |D_{\tripleTau}\bm{w}(\xi)| \cdot |D_{\tripleSig}\bm{w}(\xi)| \\
&\lesssim~ \|\bm{v}\|_{\a_1} \cdot \sum_{\xi \in \mathcal{L}}\sum_{\tripleTau \in \mathcal{R}}\sum_{\tripleSig \in \mathcal{R}} |D_{\tripleTau}\bm{w}(\xi)| \cdot |D_{\tripleSig}\bm{w}(\xi)|
\lesssim~  \|\bm{v}\|_{\a_1}  \|\bm{w}\|_{\a_1}^2.
\end{align*}
Consequently, $\bar{\mathcal{E}}^\a_{\rm hom}(\bm{u})$ is twice differentiable with
\[
\<\delta^2 \bar{\mathcal{E}}^\a_{\rm hom}(\bm{u})\bm{v},\bm{w}\> = \sum_{\xi \in \mathcal{L}}\sum_{\triple \in \mathcal{R}}\sum_{\tripleTau \in \mathcal{R}} V_{,\triple\tripleTau}(D\bm{u}(\xi)): D_{\triple}(\bm{v}(\xi)):D_{\tripleTau}(\bm{w}(\xi)).
\]

In a similar fashion, a Taylor expansion and a bound on the fourth derivative of $V$ can be used to show that $\bar{\mathcal{E}}^\a_{\rm hom}(\bm{u})$  is three times differentiable with
\begin{align*}
&\<\delta^3 \bar{\mathcal{E}}^\a_{\rm hom}(\bm{u})[\bm{v},\bm{w},\bm{z}]\> = \\
&\sum_{\xi \in \mathcal{L}}\sum_{\triple \in \mathcal{R}}\sum_{\tripleTau \in \mathcal{R}}\sum_{\tripleSig \in \mathcal{R}} V_{,\triple\tripleTau\tripleSig}(D\bm{u}(\xi))[D_{\triple}(\bm{v}(\xi)),D_{\tripleTau}(\bm{w}(\xi)), D_{\tripleSig}(\bm{z}(\xi))].
\end{align*}
Now for $\bm{u} \in \bm{\mathcal{U}}_0$, we see that $\mathcal{E}^\a_{\rm hom}(\bm{u})$ is well defined (finite) and $\mathcal{E}^\a_{\rm hom}(\bm{u}) = \bar{\mathcal{E}}^\a_{\rm hom}(\bm{u})$ due to~\eqref{ostrich1}.  Since $\bm{\mathcal{U}}_0$ is dense in $\bm{\mathcal{U}}$, it follows that $\bar{\mathcal{E}}^\a_{\rm hom}$ is the unique, continuous extension of $\mathcal{E}^\a_{\rm hom}$ to $\bm{\mathcal{U}}$, which we have also proven to be ${\rm C}^3$ on $\bm{\mathcal{U}}$.
This completes the proof of Theorem~\ref{well_defined}.

% \end{proof}

\subsection{Lattice Stability}
Here we prove that if there exists any displacement $\bm{u} \in \bm{\mathcal{U}}$ such that
\[
\<\delta^2\mathcal{E}^\a(\bm{u})\bm{v},\bm{v}\> \geq~ \gamma_\a\|\bm{v}\|_{\a_1}^2 , \quad \forall \, \bm{v} \in \bm{\mathcal{U}}_0,
\]
then the stability assumption of Assumption~\ref{coercive} is met.

\begin{lemma}\label{freeStable}
Suppose that there exists a displacement $\bm{u} \in \bm{\mathcal{U}}$ such that
\[
\<\delta^2\mathcal{E}^\a(\bm{u})\bm{v},\bm{v}\> \geq~ \gamma_\a\|\bm{v}\|_{\a_1}^2 , \quad \forall \, \bm{v} \in \bm{\mathcal{U}}_0.
\]
Then the reference configuration satisfies~\eqref{eq:stab-hom}
\end{lemma}

\begin{proof}
The proof is a straightforward extension of~\cite[Lemma 2.2]{Ehrlacher2013}.  Fix a test pair $\bm{v}$ and let $r$ be large enough so that $D\bm{v}$ has support in the ball of radius $r$.  Our goal is to find a suitable sequence of test pairs $\bm{v}_n$ which satisfy
\begin{align*}
\lim_{n \to \infty}\<\delta^2\mathcal{E}^\a(\bm{u}^\infty)\bm{v}_n,\bm{v}_n\>  = \<\delta^2\mathcal{E}^\a_{\rm hom}(0)\bm{v},\bm{v}\>.
\end{align*}
Take $\xi_n \in \mathcal{L}$ such that $|\xi_n| < |\xi_{n+1}|$ and $|\xi_n| \to \infty$, and further define $\bm{v}_n(\xi) = \bm{v}(\xi - \xi_n)$, which shifts the support of $D\bm{v}_n$ to $B_{r}(\xi_n)$.  Consequently,
\begin{align*}
\gamma_\a\|\bm{v}\|_\a^2 \leq~& \lim_{n \to \infty}\<\delta^2\mathcal{E}^\a(\bm{u})\bm{v}_n,\bm{v}_n\> =~ \lim_{n \to \infty}\sum_{\xi \in \mathcal{L}}\<\delta^2 V_\xi(D\bm{u})D\bm{v}_n(\xi), D\bm{v}_n(\xi)\> \\
=~& \lim_{n \to \infty}\sum_{\xi \in \mathcal{L} \cap B_r(\xi_n)}\<\delta^2 V_\xi(D\bm{u})D\bm{v}(\xi-\xi_n), D\bm{v}(\xi-\xi_n)\> \\
=~& \lim_{n \to \infty}\sum_{\xi \in \mathcal{L} \cap B_r(0)}\<\delta^2 V_{\xi+\xi_n}(D\bm{u}(\xi+\xi_n))D\bm{v}(\xi), D\bm{v}(\xi)\> \\
=~& \sum_{\xi \in \mathcal{L} \cap B_r(0)} \lim_{n \to \infty}\<\delta^2 V_{\xi+\xi_n}(D\bm{u}(\xi+\xi_n))D\bm{v}(\xi), D\bm{v}(\xi)\> \\
=~& \sum_{\xi \in \mathcal{L} \cap B_r(0)}\<\delta^2 V(0)D\bm{v}(\xi), D\bm{v}(\xi)\>,
\end{align*}
by virtue of $V_\xi(D\bm{u}(\xi+\xi_n)) \to V(0)$ in $\ell^\infty$ which itself
is due to $D\bm{u}^\infty \in \ell^2$.
\end{proof}

\subsection{Proof of (\ref{eq:prf-equil_shift_cor})}
\label{sec:cb_at}
\begin{proof}
Without loss of generality, we assume that
\[
\triple \in \mathcal{R} \quad \text{if and only if} \quad
         (-\rho\beta\alpha) \in \mathcal{R}.
\]
This condition can always be met by enlarging the interaction range if necessary.

To prove \eqref{eq:prf-equil_shift_cor}, we then observe that
	\begin{align*}
	&\< \mathcal{E}^\a_{\rm hom}(0), \bm{v} \> \\
	&=~ \sum_{(\rho\gamma\beta) \in \mathcal{R}} \hat{V}_{,(\rho\gamma\beta)}(D\bm{y}(\zeta)) \cdot \big[v_\beta(\zeta + \rho) - v_\gamma(\zeta)\big] + \sum_{(\rho\beta\gamma) \in \mathcal{R}} \hat{V}_{,(\rho\beta\gamma)}(D\bm{y}(\zeta)) \cdot \big[v_\gamma(\zeta + \rho) - v_\beta(\zeta)\big] \\
	&-~ \sum_{(\rho\gamma\gamma) \in \mathcal{R}} \hat{V}_{,(\rho\gamma\gamma)}(D\bm{y}(\zeta)) \cdot \big[v_\gamma(\zeta + \rho) - v_\gamma(\zeta)\big] +~ \sum_{\substack{\tripleSig \in \mathcal{R} \\ \sigma \neq 0}}\sum_{\triple \in \mathcal{R}}\hat{V}_{,\triple}(D\bm{y}(\zeta + \sigma)) \cdot D_{\triple} \bm{v}(\zeta + \sigma) \\
	&=~ -\sum_{(\rho\gamma\beta) \in \mathcal{R}} \hat{V}_{,(\rho\gamma\beta)}(D\bm{y}(\zeta)) + \sum_{(0\beta\gamma) \in \mathcal{R}} \hat{V}_{,(0\beta\gamma)}(D\bm{y}(\zeta))  - \sum_{(\rho\gamma\gamma) \in \mathcal{R}} \hat{V}_{,(\rho\gamma\gamma)}(D\bm{y}(\zeta)) \\
	&+~ \sum_{(\rho\gamma\gamma) \in \mathcal{R}} \hat{V}_{,(\rho\gamma\gamma)}(D\bm{y}(\zeta)) +~ \sum_{\substack{\tripleSig \in \mathcal{R} \\ \sigma \neq 0}}\sum_{\triple \in \mathcal{R}}\hat{V}_{,\triple}(D\bm{y}(\zeta + \sigma)) \cdot \big[v_\beta(\zeta + \sigma + \rho) - v_\alpha(\zeta + \sigma) \big] \\
	&=~ -\sum_{(\rho\gamma\beta) \in \mathcal{R}} \hat{V}_{,(\rho\gamma\beta)}(D\bm{y}(\zeta)) + \sum_{(0\beta\gamma) \in \mathcal{R}} \hat{V}_{,(0\beta\gamma)}(D\bm{y}(\zeta)) + \sum_{\substack{(\sigma\beta\gamma) \in \mathcal{R} \\ \sigma \neq 0}}\hat{V}_{,(-\sigma\beta\gamma)}(D\bm{y}(\zeta + \sigma)) \\
	&=~ -\sum_{(\rho\gamma\beta) \in \mathcal{R}} \hat{V}_{,(\rho\gamma\beta)}(D\bm{y}(\zeta)) + \sum_{(\rho\beta\gamma) \in \mathcal{R}} \hat{V}_{,(\rho\beta\gamma)}(D\bm{y}(\zeta)) \\
	&=~ -\sum_{(\rho\gamma\beta) \in \mathcal{R}} \hat{V}_{,(\rho\gamma\beta)}\big((\mG\sigma + p_\chi - p_\iota)_{\tripleSig \in \mathcal{R}} \big) + \sum_{(\rho\beta\gamma) \in \mathcal{R}} \hat{V}_{,(\rho\beta\gamma)}\big((\mG\sigma + p_\chi - p_\iota)_{\tripleSig \in \mathcal{R}} \big).
	\end{align*}
Meanwhile, straightforward computations yield
\[
\partial_{p_\gamma} \hat{W}(\mG, \bm{p}) = \sum_{(\rho\beta\gamma) \in \mathcal{R}} \hat{V}_{,(\rho\beta\gamma) }\big((\mG\sigma + p_\chi - p_\iota)_{\tripleSig \in \mathcal{R}} \big) - \sum_{(\rho\gamma\beta) \in \mathcal{R}}\hat{V}_{,(\rho\gamma\beta)}\big((\mG\sigma + p_\chi - p_\iota)_{\tripleSig \in \mathcal{R}} \big).
\]
\end{proof}

\subsection{Proof of (\ref{claimant})}
\label{sec:proof-claimant}
Applying the chain rule, and repeatedly using the fact that $\mG$ satisfies
$\partial_{\bm{p}} \hat{W}((\mG,\bm{p})) = 0$,  we obtain
\[
\partial^2_{\mG} \bar{W}(\mG) = \partial_{\mG\mG}^2 W(\mG,\bm{p}) - \partial^2_{\mG \bm{p}} W(\mG,\bm{p}) [\partial^2_{\bm{p}\bm{p}}W(\mG, \bm{p})]^{-1} \partial^2_{\bm{p}\mG} W(\mG,\bm{p}).
\]
From straightforward computations, we have
\begin{align*}
\partial^2_{\bm{p}\bm{p}}W(\mG, \bm{p}) =~& J_{\bm{p}\bm{p}} \\
\partial^2_{\mG_{mn} \bm{p}_\beta^l} W(\mG,\bm{p}) =~& \sum_{\triple \in \mathcal{R}}\sum_{\tripleTau\in\mathcal{R}} V_{,\triple\tripleTau}^{lm}(0)\tau_n - \sum_{(\rho\beta\alpha) \in \mathcal{R}}\sum_{\tripleTau\in\mathcal{R}} V_{,(\rho\beta\alpha)\tripleTau}^{lm}(0)\tau_n \\
\partial^2_{\mG_{mn} \mG_{rs}} =~& \sum_{\triple\in \mathcal{R}} \sum_{\tripleTau \in \mathcal{R}} V_{,\triple\tripleTau}^{mr}(0)\rho_n\tau_s
\end{align*}
so that
\begin{align*}
&\int_{\mathbb{R}^d} {\rm A}:\nabla Z:\nabla Z\, dx =~  \int_{\mathbb{R}^d} \partial_{\mG\mG}^2 W(\mG,\bm{p}):\nabla Z:\nabla Z\, dx \\
&\quad- \int_{\mathbb{R}^d} \partial^2_{\mG \bm{p}} W(\mG,\bm{p}) [\partial^2_{\bm{p}\bm{p}}W(\mG, \bm{p})]^{-1} \partial^2_{\bm{p}\mG} W(\mG,\bm{p}):\nabla Z:\nabla Z\, dx \\
&=~ \int_{\mathbb{R}^d}\partial^2_{\mG_{mn} \mG_{rs}} W(\mG,\bm{p})\frac{\partial}{\partial x_n}Z_{m}(x) \frac{\partial}{\partial x_s}Z_{r}(x) \, dx \\
&\quad- \int_{\mathbb{R}^d} \partial^2_{\mG_{rs} \bm{p}^\alpha_i} W(\mG,\bm{p}) [J_{\bm{p}\bm{p}}]^{-1}_{\alpha i \beta j} \partial^2_{\bm{p}^\beta_j\mG_{mn}} W(\mG,\bm{p}) \frac{\partial}{\partial x_n}Z_{m}(x) \frac{\partial}{\partial x_s}Z_{r}(x)\, dx \\
&=~ \int_{\mathbb{R}^d}4\pi^2 \partial^2_{\mG_{mn} \mG_{rs}} W(\mG,\bm{p})k_n k_s \hat{Z}^*_{m}(k)\hat{Z}_{r}(x) \, dk \\
&\quad- \int_{\mathbb{R}^d} 4\pi^2 \hat{Z}^*_{m}(k) \partial^2_{\mG_{rs} \bm{p}^\alpha_i} W(\mG,\bm{p})k_s [J_{\bm{p}\bm{p}}]^{-1}_{\alpha i \beta j} \partial^2_{\bm{p}^\beta_j\mG_{mn}} W(\mG,\bm{p}) k_n \hat{Z}_{r}(k)\, dk \\
&=~ \int_{\mathbb{R}^d} \hat{Z}^*_{m}(k) J_{00}^{mr}(k) \hat{Z}_{r}(k) \, dk - \int_{\mathbb{R}^d}  \hat{Z}^*_{m}(k) [ J_{0\bm{p}} J_{\bm{p}\bm{p}}^{-1} J_{\bm{p}0}]_{mr} \hat{Z}_{r}(k)\, dk \\
&=~ \int_{\mathbb{R}^d} \hat{Z}^*(k) M(k) \hat{Z}(k) \, dk ~\gtrsim \|\nabla Z\|^2_{L^2(\mathbb{R}^d)}.
\end{align*}
This completes the proof of \eqref{claimant}.

\subsection{Norm Equivalence}\label{norm_equiv}

\begin{lemma}
The norms defined for $\bm{v} = (Z,\bm{q})$ by
\[
\|\bm{v}\|_{\a_3}^2 = \|(Z, \bm{q})\|_{\a_3}^2 := \|2\pi|k|\hat{Z}\|_{L^2(\mathcal{B})}^2 + \sum_{\alpha = 1}^{S-1} \|\hat{q}_\alpha\|^2_{L^2(\mathcal{B})}
\]
and
\[
\| \bm{v}\|_{\a_2}^2 := \|\nabla IZ\|_{L^2(\mathbb{R}^d)}^2 + \sum_{\alpha}\|Iq_\alpha\|_{L^2(\mathbb{R}^d)}^2.
\]
are equivalent on $\bm{\mathcal{U}}$.
\end{lemma}

\begin{proof}
Note
\begin{align*}
\sum_{i =1}^d\sum_{\xi \in \mathcal{L}}|D_{e_i}v_0(\xi)|^2 \lesssim~ \|\nabla IZ\|_{L^2(\mathbb{R}^d)}^2 \lesssim~ \sum_{i =1}^d\sum_{\xi \in \mathcal{L}}|D_{e_i}v_0(\xi)|^2
\end{align*}
and
\begin{align*}
\sum_{i =1}^d\sum_{\xi \in \mathcal{L}}|D_{e_i}v_0(\xi)|^2 =~& \sum_{i=1}^d\int_{\mathcal{B}} \widehat{D_{e_i} Z}^* \widehat{D_{e_i}Z} \\
=~& \sum_{i=1}^d\int_{\mathcal{B}} 4\sin^2(\pi k_i) |\hat{Z}|^2(k)
\end{align*}
Since
\[
\|2\pi|k| Z\|_{L^2(\mathcal{B})}^2 \lesssim~ \sum_{i=1}^d\int_{\mathcal{B}} 4\sin^2(\pi k_i) |\hat{Z}|^2(k) \lesssim~ \|2\pi|k| \hat{Z}\|_{L^2(\mathcal{B})}^2
\]
we see that
\[
\|2\pi|k| \hat{Z}\|_{L^2(\mathcal{B})} \lesssim~ \|\nabla IZ\|_{L^2(\mathbb{R}^d)}^2 \lesssim~ \|2\pi|k| \hat{Z}\|_{L^2(\mathcal{B})}.
\]
Similarly,
\begin{align*}
\int_{\mathcal{B}} |\hat{q_\alpha}|^2 = \|q_\alpha \|_{\ell^2(\mathcal{L})} \lesssim~ \|Iq_\alpha\|^2_{L^2(\mathbb{R}^d)} \lesssim~ \|q_\alpha \|_{\ell^2(\mathcal{L})} = \int_{\mathcal{B}} |\hat{q_\alpha}|^2.
\end{align*}
\end{proof}

%%%%%%%%%%%%%%%%%%%%%%%%%%%%%%%%%%%%%%%%%%%%%%%%%%%%%%%%%%%%%%%%%%%%%%
\bibliographystyle{plain}	% (uses file "plain.bst")
\bibliography{myrefs}		% expects file "myrefs.bib"

\begin{thebibliography}{10}

\bibitem{alden2013strain}
J.~Alden, A.~Tsen, P.~Huang, R.~Hovden, L.~Brown, J.~Park, D.~Muller, and
  P.~McEuen.
\newblock Strain solitons and topological defects in bilayer graphene.
\newblock {\em Proceedings of the National Academy of Sciences},
  110(28):11256--11260, 2013.

\bibitem{bacon1980}
D.J. Bacon, D.M. Barnett, and R.O. Scattergood.
\newblock Anisotropic continuum theory of lattice defects.
\newblock {\em Progress in Materials Science}, 23:51 -- 262, 1980.

\bibitem{born1954}
M.~Born and K.~Huang.
\newblock {\em Dynamical Theory of Crystal Lattices}.
\newblock Clarendon Press, first edition, 1954.

\bibitem{cauchy}
A.L. Cauchy.
\newblock De la pression ou la tension dans un systeme de points materiels.
\newblock In {\em Exercices de Mathematiques}. 1828.

\bibitem{cazeaux2016}
P.~Cazeaux, M.~Luskin, and E.~B. Tadmor.
\newblock Analysis of rippling in incommensurate one-dimensional coupled
  chains.
\newblock {\em ArXiv e-prints}, 2016.

\bibitem{weinan2007cauchy}
W.~E and P.~Ming.
\newblock Cauchy--born rule and the stability of crystalline solids: static
  problems.
\newblock {\em Archive for Rational Mechanics and Analysis}, 183(2):241--297,
  2007.

\bibitem{Ehrlacher2013}
V.~Ehrlacher, C.~Ortner, and A.V. Shapeev.
\newblock {Analysis of Boundary Conditions for Crystal Defect Atomistic
  Simulations}.
\newblock {\em ArXiv e-prints}, June 2013.
\newblock 1306.5334.

\bibitem{eshelby1956}
J.D. Eshelby.
\newblock The continuum theory of lattice defects.
\newblock volume~3 of {\em Solid State Physics}, pages 79--144. Academic Press,
  1956.

\bibitem{flinn1962}
P.~Flinn and A.~Maradudin.
\newblock Distortion of crystals by point defects.
\newblock {\em Annals of Physics}, 18(1):81 -- 109, 1962.

\bibitem{glebov2012}
O.A. Glebov and M.A. Krivoglaz.
\newblock {\em X-Ray and Neutron Diffraction in Nonideal Crystals}.
\newblock Springer Berlin Heidelberg, 2012.

\bibitem{gupta2002powder}
S.P.S. Gupta and Indian~Association for the Cultivation~of Science.
\newblock {\em Powder diffraction : proceedings of the II International School
  on Powder Diffraction ; January 20 - 23, 2002, IACS, Kolkata, India ; (as
  part of 125 years of celebration)}.
\newblock Allied Publishers, 2002.

\bibitem{hardy1960}
J.R. Hardy.
\newblock A theoretical study of point defects in the rocksalt structure
  substitutional k+ in nacl.
\newblock {\em Journal of Physics and Chemistry of Solids}, 15(1):39 -- 49,
  1960.

\bibitem{hudson2012}
T.~Hudson and C.~Ortner.
\newblock On the stability of {B}ravais lattices and their {C}auchy--{B}orn
  approximations.
\newblock {\em M2AN Math. Model. Numer. Anal.}, 46:81--110, 2012.

\bibitem{hudson2013}
T.~Hudson and C.~Ortner.
\newblock Existence and stability of a screw dislocation under anti-plane
  deformation.
\newblock {\em Arch. Ration. Mech. Anal.}, 213(3):887--929, 2014.

\bibitem{kanzaki1957}
H.~Kanzaki.
\newblock Point defects in face-centred cubic lattice—i distortion around
  defects.
\newblock {\em Journal of Physics and Chemistry of Solids}, 2(1):24 -- 36,
  1957.

\bibitem{koten2013}
B.~Van Koten and C.~Ortner.
\newblock Symmetries of 2-lattices and second order accuracy of the
  cauchy--born model.
\newblock {\em SIAM Multiscale Modelling and Simulation}, 11:615--634, 2013.

\bibitem{krivoglaz1969}
M.A. Krivoglaz.
\newblock {\em Theory of X-ray and thermal-neutron scattering by real
  crystals}.
\newblock Plenum Press, 1969.

\bibitem{rodin2002}
P.~Martinsson and G.~Rodin.
\newblock Asymptotic expansions of lattice green's functions.
\newblock {\em Proceedings of the Royal Society of London A: Mathematical,
  Physical and Engineering Sciences}, 458(2027):2609--2622, 2002.

\bibitem{Morrey}
C.~Morrey.
\newblock {\em Multiple Integrals in the Calculus of Variations}.
\newblock Springer, 1966.

\bibitem{novoselov2005}
K.~S. Novoselov, D.~Jiang, F.~Schedin, T.~J. Booth, V.~V. Khotkevich, S.~V.
  Morozov, and A.~K. Geim.
\newblock Two-dimensional atomic crystals.
\newblock {\em Proceedings of the National Academy of Sciences of the United
  States of America}, 102(30):10451--10453, 2005.

\bibitem{olsonUnpub}
D.~Olson, X.~Li, C.~Ortner, and B.~Van~Koten.
\newblock {\em Unpublished Manuscript}, 2016.

\bibitem{olson2015}
D.~Olson, A.~Shapeev, P.~Bochev, and M.~Luskin.
\newblock Analysis of an optimization-based atomistic-to-continuum coupling
  method for point defects.
\newblock {\em ESAIM: M2AN}, 50(1):1--41, 2016.

\bibitem{suli2012}
C.~Ortner and E.~S\"{u}li.
\newblock A note on linear elliptic systems on $\mathbb{R}^d$.
\newblock {\em ArXiv e-prints}, 2012.
\newblock 1202.3970.

\bibitem{theil2012}
C.~Ortner and F.~Theil.
\newblock Justification of the {C}auchy--{B}orn approximation of
  elastodynamics.
\newblock {\em Arch. Ration. Mech. Anal.}, 207:1025--1073, 2013.

\bibitem{phillips2001}
R.~Phillips.
\newblock {\em Crystals, defects and microstructures: modeling across scales}.
\newblock Cambridge University Press, 2001.

\bibitem{rudin1987}
W.~Rudin.
\newblock {\em Real and complex analysis}.
\newblock McGraw-Hill, 1987.

\bibitem{smith1992}
R.~Smith.
\newblock Some interlacing properties of the schur complement of a hermitian
  matrix.
\newblock {\em Linear Algebra and Its Applications}, 177:137--144, 1992.

\bibitem{tewary1973}
V.~Tewary.
\newblock Green-function method for lattice statics.
\newblock {\em Advances in Physics}, 22(6):757--810, 1973.

\bibitem{trefethen2000}
L.N. Trefethen.
\newblock {\em Spectral Methods in MATLAB}.
\newblock Society for Industrial and Applied Mathematics, 2000.

\bibitem{wallace1998}
D.~Wallace.
\newblock {\em Thermodynamics of Crystals}.
\newblock Dover, 1998.

\bibitem{zhang2006schur}
F.~Zhang.
\newblock {\em The Schur Complement and Its Applications}.
\newblock Numerical Methods and Algorithms. Springer US, 2006.

\end{thebibliography}

\end{document}